\documentclass[11pt]{amsart}

\title{Residually finite groups with uniformly almost flat quotients}
\author{David Guo and Matthew Tointon}
\address{School of Mathematics, University of Bristol, United Kingdom}
\email{h.guo@bristol.ac.uk}
\email{m.tointon@bristol.ac.uk}
\usepackage{hyperref}
\usepackage{stmaryrd}
\usepackage{graphicx}
\usepackage{framed}
\usepackage{amsmath}
\usepackage{amssymb}
\usepackage{amsfonts}
\usepackage{mathrsfs}
\usepackage{array}
\usepackage{enumerate}
\usepackage{comment}
\usepackage{amsthm}  
\usepackage{mathtools}
\usepackage{physics}
\usepackage[dvipsnames]{xcolor}
\usepackage{braket}
\usepackage[capitalise,nameinlink ,noabbrev]{cleveref} 
\usepackage[top=3.5cm, bottom=3.5cm, left=2.5cm, right=2.5cm]{geometry}
\usepackage{soul}
\usepackage{nicematrix}

\DeclareMathOperator{\diam}{diam}
\DeclareMathOperator{\Aut}{Aut}
\DeclareMathOperator{\Inn}{Inn}
\DeclareMathOperator{\res}{res}
\DeclareMathOperator{\GL}{GL}
\DeclareMathOperator{\Tor}{Tor}
\DeclareMathOperator{\ord}{ord}

\DeclareMathOperator{\Fit}{Fit}
\providecommand{\id}{\mathrm{id}}
\DeclareMathOperator{\csubg}{\underset{\mathrm{char}}{\normal}}

\providecommand{\glnz}{\GL(n, \Z)}

\providecommand{\R}{\mathbb{R}}
\providecommand{\C}{\mathbb{C}}
\providecommand{\GE}{\mathcal{G}}
\renewcommand{\cp}{\mathcal{P}}
\renewcommand{\H}{\mathcal{H}}
\renewcommand{\ge}{\mathcal{G}}
\providecommand{\F}{\mathbb{F}}
\providecommand{\Fp}{\mathbb{F}_p}
\providecommand{\FP}{\Fit(P)}

\providecommand{\FKiP}{\Fit(K_i')}

\providecommand{\FPP}{\Fit(P')}

\providecommand{\ti}{\tilde}
\providecommand{\PFP}{P/\Fit(P)}
\providecommand{\N}{\mathbb{N}}
\providecommand{\Z}{\mathbb{Z}}
\providecommand{\Q}{\mathbb{Q}}
\providecommand{\T}{\Tor}
\providecommand{\A}{\Aut^+}
\renewcommand{\a}{\alpha}
\renewcommand{\t}{\theta}
\renewcommand{\u}{\underline}
\renewcommand{\b}{\beta}
\renewcommand{\v}{\underline{v}}
\renewcommand{\u}{\underline{u}}

\providecommand{\w}{\underline{w}}
\providecommand{\V}{\underline{V}}
\providecommand{\f}{_{f.i}}

\providecommand{\iso}{ \cong}

\renewcommand{\l}{ \lambda}

\providecommand{\d}{ \delta}
\providecommand{\g}{ \gamma}
\providecommand{\G}{ \Gamma}

\providecommand{\p}{ \phi}
\renewcommand{\P}{ \Phi}
\renewcommand{\Pr}{ \mathcal{P}}
\renewcommand{\sb}{ \bar{\psi}}

\providecommand{\sm}{\sum_{i=1}^m}

\providecommand{\dgk}{\diam_S (G/K)}

\providecommand{\dgh}{\diam_S (G/H)}

\providecommand{\subg}{\leqslant}

\providecommand{\n}[1]{\norm{#1}} 
\renewcommand{\i}{^{-1}}
\providecommand{\G}[1]{G_{#1}} 
\providecommand{\d}[2]{\diam(G/H)} 
\providecommand{\F}[1]{\Fit{#1}} 
\providecommand{\nsubg}{\normal}
\providecommand{\containsn}{\trianglerighteq} 
\providecommand{\ch}{\operatorname{char}}
\providecommand{\rt}{\rtimes} 
\providecommand{\subeq}{\subseteq} 

\providecommand{\KF}[1]{K_{#1}'/\FP} 

\providecommand{\Gm}{G_{i-1}} 
\providecommand{\Gp}{G_{i+1}}

\providecommand{\Pp}{P_{i+1}} 
\providecommand{\Pi}{P_{i}}

\providecommand{\Zn}{\Z^n} 

\newcommand{\la}{\langle}
\newcommand{\ra}{\rangle}
\newcommand{\normal}{\trianglelefteqslant}
\renewcommand{\ge}{\geqslant}
\renewcommand{\le}{\leqslant}
\renewcommand{\geq}{\geqslant}
\renewcommand{\leq}{\leqslant}
\newcommand{\eps}{\varepsilon}

\newcommand{\nn}{\mathfrak{n}}
\newcommand{\Span}{\mathrm{span}}

\numberwithin{equation}{section}

\theoremstyle{plain}
\crefname{prop}{Proposition}{Propositions}
\crefname{thm}{Theorem}{Theorems}
\newtheorem{thm}{Theorem}[section] 
\newtheorem*{thm*}{Theorem}
\newtheorem{lem}[thm]{Lemma} 
\newtheorem{prop}[thm]{Proposition} 
\newtheorem{cor}[thm]{Corollary} 

\newtheorem{conjecture}[thm]{Conjecture}

\theoremstyle{definition}
\newtheorem*{remark*}{Remark}
\newtheorem{remark}[thm]{Remark} 

\crefname{lem}{Lemma}{Lemmas}

\theoremstyle{definition}

\theoremstyle{remark}

\makeatletter\@addtoreset{case}{thm}\makeatother

\DeclarePairedDelimiterX{\inp}[2]{\langle}{\rangle}{#1, #2}

\usepackage[numbers]{natbib}
\usepackage{cite}
\usepackage{url}

\begin{document}  
\maketitle

\begin{abstract}
We show that if all the finite coset spaces of a polycyclic group have diameter bounded uniformly below by a polynomial in their size then the group is virtually nilpotent. We obtain the same conclusion for a finitely generated residually torsion-free nilpotent group under the weaker assumption that the finite quotient groups have diameter bounded uniformly below by a polynomial in their size. This extends work of Khukhro and Valette.
\end{abstract}

\tableofcontents

\section{Introduction}
A fundamental theorem of Gromov \citep{gromov} states that a finitely generated group has polynomial growth if and only if it is virtually nilpotent. In recent years, Kleiner \citep{kleiner} and Ozawa \citep{ozawa} have given shorter proofs of this theorem, and Shalom and Tao \citep{st} and Breuillard, Green and Tao \citep{bgt} have obtained refined statements giving stronger, more quantitative conclusions and requiring weaker hypotheses in which only a single ball in the group needs to have polynomial size.

To be more precise, let $G$ be a group, and let $S$ be a finite symmetric generating set containing the identity, so that the ball of radius $r\in\N$ in $G$ with respect to $S$ is exactly the set $S^r=\{s_1\cdots s_r:s_i\in S\}$. Breuillard, Green and Tao showed that for each $k\in\N$ there exists $N=N(k)>0$ such that if $|S^n|\le n^k|S|$ for some $n\ge N$ then there are normal subgroups $H,\Gamma\normal G$ with $[G:\Gamma]$ bounded, $H\subeq \G \cap S^n$  and $\G/H$ nilpotent of bounded rank and class.

In the case where $G$ is finite, a natural value of $n$ at which to apply this result is the \emph{diameter} of $G$, defined by $\diam_S(G)=\min\{n\in\N:S^n=G\}$. When $n = \diam_S(G)$, the condition $|S^n|\le n^k|S|$ translates to the condition
\begin{equation}\label{eq:af}
\diam_S(G)\ge\left(\frac{|G|}{|S|}\right)^\alpha
\end{equation}
with $\alpha=1/k$. Breuillard and the second author \citep{nil} called groups satisfying \eqref{eq:af} \emph{$\alpha$-almost flat}, and proved that such a group necessarily admits a quotient with an abelian subgroup of bounded index and diameter comparable to that of the group. The sharpest version of this result currently available is the following theorem of Tessera and the second author \citep{ttBalls}, which gives in particular the optimal bound on the rank of the abelian subgroup.

\begin{thm}\label{thm:bt}
For every non-negative integer $d$ there exists $A=A(d)>0$ such that if $G$ is a finite group with symmetric generating set $S$ containing the identity such that
\[
\diam_S(G)\ge A\left(\frac{|G|}{|S|}\right)^\frac2{d+2}
\]
then there is a normal subgroup $H\normal G$ contained in $S^{O_d(\diam_S(G)^{1/2})}$ such that $G/H$ has an abelian subgroup of rank at most $d$ and index at most $(2d)!$.
\end{thm}

Here and throughout this paper we adopt a standard version of asymptotic notation, in which $O(X)$ means a quantity bounded above by a constant multiple of $X$. If the implied constant depends on some other parameters $\lambda_1,\ldots,\lambda_k$, we indicate this with subscripts, e.g. $O_{\lambda_1,\ldots,\lambda_k}(X)$.

Khukhro and Valette \citep{box} went on to study almost flatness as a residual property. Given $\alpha\in(0,1]$, define a finitely generated group $G$ to be \emph{residually $\alpha$-almost flat} if there exists $\eps>0$ such that for every non-identity element $g\in G$ there exists a finite-index normal subgroup $N_g\normal G$ such that $g\notin N_g$ and such that
\[
\diam_S(G/N_g)\ge\eps[G:N_g]^\alpha.
\]
Define $G$ to be \emph{strongly residually $\alpha$-almost flat} if there exists $\eps>0$ such that for every finite subset $A\subseteq G$ there exists a finite-index normal subgroup $N_A\normal G$ such that $A\cap N_A\subseteq\{1\}$ and such that
\[
\diam_S(G/N_A)\ge\eps[G:N_A]^\alpha.
\]
Here and elsewhere, given a subgroup $H\le G$ we define $\diam_S(G/H)=\min \{n\in \N : S^n H = G\}$. We often drop the subscript $S$ when it is clear from the context. Define $G$ to be \emph{residually almost flat} if there exists some $\alpha\in(0,1]$ such that $G$ is residually $\alpha$-almost flat, and \emph{strongly residually almost flat} if there exists some $\alpha\in(0,1]$ such that $G$ is strongly residually $\alpha$-almost flat. The property of being residually $\alpha$-almost flat might a priori depend on the choice of generating set, but it is easy to verify that it does not (see \cref{lem:indep.gen.set}). Note that a residually almost flat group is by definition residually finite.
\begin{thm}[Khukhro--Valette \citep{box}]\label{thm:kv.detailed}\label{thm:kv}
Let $G$ be an infinite residually finite group. Then the following statements are equivalent.
\begin{enumerate}
\item The group $G$ is residually almost flat.\label{item:raf}
\item The group $G$ is strongly residually almost flat.\label{item:box}
\item The group $G$ admits a virtual surjection to $\Z$.\label{item:indic}
\end{enumerate}
\end{thm}
By a \emph{virtual surjection} onto $\Z$ we mean a surjective homomorphism from a finite-index subgroup of $G$ to $\Z$. \cref{thm:kv} stated slightly differently from in Khukhro and Valette's paper, but in \cref{sec:kv} we give details of how it follows from their argument. The implication \eqref{item:raf} $\implies$ \eqref{item:indic} is essentially due to Breuillard and the second author \citep{nil}.

Although \cref{thm:kv} shows that a group with a virtual surjection to $\Z$ is residually almost flat, such a group can also admit quotients with extremely small diameter; indeed the free group $F_k$ admits every $k$-generated group as a quotient. On the other hand, it is easy to see that if $G$ is a group of polynomial growth of degree $d$ then there exists $\eps>0$ such that \emph{every} finite quotient $G/N$ of $G$ satisfies $\diam_S(G/N)\ge\eps[G:N]^{1/d}$ (see \cref{lem:poly.growth}). The purpose of this paper is to investigate whether the groups of polynomial growth are in fact the only groups with this property, and more generally to investigate what algebraic restrictions this property places on a group.

To be more precise, given $\alpha\in(0,1]$, let us say that a residually finite group $G$ with a finite generating set $S$ has \emph{uniformly $\alpha$-almost flat quotients} if there exists $\eps>0$ such that for every $N\normal G$ of finite index we have
\[
\diam_S(G/N)\ge\eps[G:N]^\alpha,
\]
and that $G$ has \emph{uniformly $\alpha$-almost flat coset spaces} if there exists $\eps>0$ such that for every $H\le G$ of finite index we have
\[
\diam_S(G/H)\ge\eps[G:H]^\alpha.
\]
Say that $G$ has \emph{uniformly almost flat quotients} if there exists some $\alpha\in(0,1]$ such that $G$ has uniformly $\alpha$-almost flat quotients, and that $G$ has \emph{uniformly almost flat coset spaces} if there exists some $\alpha\in(0,1]$ such that $G$ has uniformly $\alpha$-almost flat coset spaces. As with residual almost flatness, these definitions do not depend on the choice of generating set (see \cref{lem:indep.gen.set}). Note also that they are trivially inherited by quotients of $G$, a fact that we will use repeatedly without further mention.

In the case $\alpha=1$, these properties are highly restrictive, as was shown by Khukhro and Valette as follows.
\begin{thm}[Khukhro--Valette {\citep[Theorem 3]{box}}]\label{thm:kv.1}
Let $G$ be an infinite residually finite group. Then the following are equivalent.
\begin{enumerate}
\item The group $G$ is strongly residually $1$-almost flat.\label{item:r1af}
\item The group $G$ has uniformly $1$-almost flat quotients.\label{item:1afq}
\item The group $G$ has uniformly $1$-almost flat coset spaces.\label{item:1afcs}
\item The group $G$ is virtually cyclic.\label{item:vz}
\end{enumerate}
\end{thm}
\cref{thm:kv.1} is phrased slightly differently from in Khukhro and Valette's paper, but we prove the version stated above in \cref{sec:kv}, following their proof.
\begin{remark*}
\cref{thm:kv.1} cannot be improved to say that a residually $1$-almost flat group is virtually cyclic, since for example every finitely generated abelian group is residually $1$-almost flat. Nonetheless, by \cref{thm:kv.detailed} a residually $1$-almost flat group is strongly residually $\alpha$-almost flat for \emph{some}~$\alpha$.
\end{remark*}

For smaller values of $\alpha$ the equivalence between being strongly residually $\alpha$-almost flat, having $\alpha$-almost flat quotients, and having $\alpha$-almost flat coset spaces breaks down. This is most striking for strong residual almost flatness, where Khukhro and Valette show that for any $\alpha<1$, every finitely generated residually finite group that surjects onto $\Z$ is strongly residually $\alpha$-almost flat \citep[Proposition 5]{box}. In our first result, we show that moreover having uniformly $\alpha$-almost flat quotients is not equivalent to having uniformly $\alpha$-almost flat coset spaces, as follows.
\begin{prop}\label{prop:heis}
The maximum $\alpha$ for which the Heisenberg group
\[
\left(
\begin{array}{ccc}
1&\Z&\Z\\
0&1&\Z\\
0&0&1\\
\end{array}
\right)
\]
has uniformly $\alpha$-almost flat quotients is $1/3$. The maximum $\beta$ for which the Heisenberg group has uniformly $\beta$-almost flat coset spaces is $1/4$.

\end{prop}
Our next results, which represent the main content of this paper, we show that for certain classes of groups, having uniformly almost flat quotients or coset spaces really is equivalent to having polynomial growth.
For the purposes of understanding these results, recall that finitely generated residually torsion-free nilpotent groups and virtually polycyclic groups are always residually finite.
\begin{thm}\label{thm:poly}\label{polycyclic_by_finite_group_uda_iff_vn}
Suppose $G$ is an infinite virtually polycyclic group that has uniformly $(1/d)$-almost flat coset spaces for some $d\in\N$. Then $G$ has a finite-index nilpotent subgroup $N$ such that the Hirsch lengths of $N$ and $[N,N]$ sum to at most~$d$.
\end{thm}
\begin{thm}\label{thm:resid}
Suppose $G$ is a finitely generated residually (torsion-free nilpotent) group, and let $d\in\N$. If $G$ has uniformly $(1/d)$-almost flat quotients then $G$ is nilpotent with Hirsch length at most~$d$, and if $G$ has uniformly $(1/d)$-almost flat coset spaces then $G$ is nilpotent, and the Hirsch lengths of $G$ and $[G,G]$ sum to at most~$d$.
\end{thm}
The finitely generated residually torsion-free nilpotent groups include many well-known examples of groups, such as free groups  \citep[Theorem 1]{free_surface}, surface groups \citep[Theorem 5]{free_surface}, right-angled Artin groups \citep{raagthesis},  the Torelli group of a right-angled Artin group \citep{torelli}, pure braid groups \citep{pure}, the Torelli subgroup of the mapping class group \citep{mapping} and many more (e.g. see \citep{rtfn,reflection}).

Our approach to \cref{thm:poly} is inspired by the proof of Wolf's theorem that a polycyclic group has polynomial growth if and only if it virtually nilpotent as presented for example in \citep[Chapter~14]{ggt}, but is considerably more complicated. One difficulty lies in actually constructing, in a polycyclic group of superpolynomial growth, an infinite sequence of finite-index subgroups whose coset spaces violate the uniformly $\alpha$-almost flat condition. This turns out to require various number-theoretic facts, as well as a result of Breuillard and the second author {\citep[Theorem 4.1]{nil}} that in turn relies on Breuillard, Green and Tao's deep refinement of Gromov's theorem. Another issue is that the proof of Wolf's theorem uses the fact that having polynomial growth passes to finitely generated subgroups, but a priori we do not have the analogous statement that having uniformly $\alpha$-almost flat coset spaces passes to finitely generated subgroups.

Encouraged by the above results, we tentatively conjecture that in an arbitrary finitely generated residually finite group having uniformly almost flat quotients is equivalent to having polynomial growth.
\begin{conjecture}\label{conj:main}
A finitely generated residually finite group has uniformly almost flat quotients if and only if it has uniformly almost flat coset spaces, if and only if it is virtually nilpotent. Indeed, given $d\in\N$, if a residually finite group $G$ has uniformly $(1/d)$-almost flat quotients then $G$ has a finite-index nilpotent subgroup with Hirsch length at most~$d$.
\end{conjecture}
We caution however that these uniform almost flat properties appear to be quite subtle, not least because they have both geometric and algebraic content. Indeed, whilst the almost flat condition is geometric in nature, how hard it is to have that property uniformly over all quotients depends strongly on the algebraic structure of the group. \cref{prop:heis} provides a concrete illustration of this, since the only difference between the quotient and coset-space conditions is algebraic. In a similar vein, proving the quotient version of \cref{conj:main} appears to be harder than proving the coset-space version in part because there is no obvious way of relating an arbitrary finite quotient of a finite-index subgroup of a group $G$ to a finite quotient of $G$ itself.

We can show that the qualitative conclusion of \cref{conj:main} -- i.e. that uniformly almost flat coset spaces imply virtual nilpotence -- is enough to imply the full conjecture. Indeed, given uniformly almost flat coset spaces imply virtual nilpotence, the bound on the Hirsch length of the nilpotent subgroup would follow immediately from the contrapositive of the following proposition.
\begin{prop}\label{thm:nilp}
Suppose $G$ is a finitely generated virtually nilpotent group with Hirsch length $h\in\N$. Then $G$ does not have uniformly $\alpha$-almost flat quotients for any $\alpha>1/h$.
\end{prop}
We also have an analogous result for uniformly almost flat coset spaces, which in particular gives the quantitative conclusion of \cref{thm:resid}.
\begin{prop}\label{thm:nilp.coset}
Suppose $G$ is a finitely generated virtually nilpotent group, and let $N$ be a finite-index nilpotent subgroup. Suppose $N$ has Hirsch length $h\in\N$ and $[N,N]$ has Hirsch length $h'$. Then $G$ does not have uniformly $\alpha$-almost flat coset spaces for any $\alpha>1/(h+h')$.
\end{prop}

In his forthcoming PhD thesis, the first author will show that the coset-space version of \cref{conj:main} holds for all finitely generated residually finite abelian-by-$\Z$ groups. In the special cases of the Baumslag--Solitar groups $\mathrm{BS}(1,k)$ with $k\ge2$ and wreath products $A\wr\Z$ with $A$ abelian, he will show that the full conjecture holds. Since wreath products $H\wr\Z$ with $H$ non-abelian are not residually finite \citep[Theorem 3.2]{residual}, this completely settles \cref{conj:main} for groups of this form.

We now note some further ad-hoc methods for showing that certain groups satisfy \cref{conj:main}.
\begin{cor}\label{cor:torsion}
Suppose $G$ is a finitely generated infinite residually finite torsion group. Then $G$ does not have uniformly almost flat quotients.
\end{cor}
\begin{proof}
Being torsion, $G$ does not admit a virtual surjection to $\Z$, so the result follows from \cref{thm:kv}.
\end{proof}
\begin{remark}\cref{cor:torsion} implies for example that Grigorchuk's famous group of intermediate growth \citep{grigorchuk}, which is finitely generated and residually finite, satisfies \cref{conj:main}. Grigorchuk~\citep{grigorchukTF} also constructed a torsion-free group of intermediate growth, also finitely generated and residually finite, and of course the same argument does not apply directly to this group. However, since this group admits the periodic Grigorchuk group as a quotient, we can nonetheless conclude that it also satisfies \cref{conj:main}.
\end{remark}

The relation of \cref{thm:poly} to \cref{conj:main} is analogous to the relation of Wolf's theorem to Gromov's theorem. Wolf's theorem had been known for some time before Gromov proved his theorem in general, and indeed Gromov's, Kleiner's and Ozawa's proofs of Gromov's theorem all reduce the general case to the polycyclic case, and hence to Wolf's theorem. The reduction of Gromov's theorem to Wolf's theorem proceeds broadly as follows. One first shows that if $G$ is a group with polynomial growth of degree $d$ then $G$ virtually surjects onto $\Z$. One then shows that the kernel of this surjection is finitely generated and has polynomial growth of degree at most $d-1$, and concludes by induction that $G$ must be virtually polycyclic. \cref{thm:kv} is suggestive of a similar reduction of \cref{conj:main}. Indeed, in light of \cref{thm:kv}, one  would hope to be able to deduce \cref{conj:main} from the polycycilc case and the following conjecture.
\begin{conjecture}
Let $d\in\N$. Suppose $G$ is a finitely generated infinite residually finite group with uniformly $(1/d)$-almost flat quotients, and $H$ is a finite-index subgroup admitting a quotient $H/K\cong\Z$. Then $K$ is finitely generated and has uniformly $(1/(d-1))$-almost flat quotients.
\end{conjecture}

\subsection*{Acknowledgements} 
We are indebted to James Rawson, who investigated the property of having uniformly almost flat quotients under the supervision of the second author through the Summer Research in Maths programme at the University of Cambridge. In particular, we thank him for the suggestion to consider groups with uniformly almost flat coset spaces in addition to those with uniformly almost flat quotients. The first author thanks Yves Cornulier, and the second author thanks Ana Khukhro and Alain Valette, for helpful conversations.  We also thank  Charles Cox and Alain Valette for comments on a draft, and Mark Hagen for help with the references.

\section{Basic properties of uniform almost flatness}\label{sec:basic}
The first two results of this section are essentially trivial, but of such fundamental importance that we include them nonetheless.

\begin{lem}\label{lem:indep.gen.set}
The definitions of being residually $\alpha$-almost flat, strongly residually $\alpha$-almost flat, having uniformly $\alpha$-almost flat quotients, and having uniformly $\alpha$-almost flat coset spaces do not depend on the choice of generating set.
\end{lem}
\begin{proof}
Let $G$ be a group with finite symmetric generating sets $S$ and $T$, both containing the identity. Let $m\in\N$ such that $S\subseteq T^m$. Suppose $H\le G$ has finite index, and let $\gamma=\diam_S(G/H)$. Then $G\subseteq S^\gamma H\subseteq T^{m\gamma}H$, so that $\diam_T(G/H)\le m\diam_S(G/H)$.
\end{proof}

\begin{lem}\label{lem:poly.growth}\label{poly_growth_implies_uniform_D_a}
Suppose $G$ is a group of polynomial growth of degree $d$. Then $G$ has uniformly $(1/d)$-almost flat coset spaces.
\end{lem}

\begin{proof}
Let $S$ be a finite symmetric generating set for $G$, so that $|S^n|\le Cn^d$ for some $C>0$ for all $n\in\N$. In particular, if $H\le G$ has finite index then $[G:H]\le C\diam_S(G/H)^d$.
\end{proof}

We now show that having uniformly almost flat coset spaces is equivalent to having a finite-index subgroup that has this property. We start with the following slight variant of a classical result providing a finite generating set for a finite-index subgroup.

\begin{lem} \label{generating_set_for_the_subgroup_in_terms_of_ball}
Suppose $G$ is a group with symmetric generating set $S$ containing the identity, and that $H\le G$ has finite index. Then $S^{2\diam_S(G/H)+1} \cap H$ generates $H$ with diameter at most $\diam_S(G)$.
\end{lem}

\begin{proof}
By the definition of diameter, $S ^{\diam_S(G/H)}$ contains a set of coset representations of $H$ in $G$, call this set $C$. It is shown in the proof of \citep[Lemma~7.2.2]{the_theory_of_groups} that $S^m\cap H\subseteq(CSC^{-1}\cap H)^m$ for all $m\in\N$, and hence in particular that  $CSC^{-1} \cap H \subseteq S^{2\diam_S(G/H) +1}$ generates $H$ with diameter at most $\diam_S(G)$.
\end{proof}

\begin{lem} \label{boundind_diam_G_K}
Suppose $K \le H \le G = \langle S \rangle$ with $|G:K|<\infty$, and let $T=S^{2\diam_S(G/H)+1} \cap H$ (which generates $H$ by \cref{generating_set_for_the_subgroup_in_terms_of_ball}).
Then $$\diam_T (H/K) \leq \diam_S(G/K) \leq 4 \diam_S (G/H) \diam_T (H/K).$$
\end{lem}

\begin{proof}
For the first inequality, given $h \in H$, note that $ h \in s_1 \ldots s_m K$ for some  $m \leq \diam_S(G/K)$ and $s_i \in S$. Since $K \subeq H$, $s_1 \ldots s_m  \in H$. Let $T$ be the generating set for $H$ defined in \cref{generating_set_for_the_subgroup_in_terms_of_ball}.  Applying \citep[Lemma~7.2.2]{the_theory_of_groups} once more as in the proof of the previous lemma, we can write $s_1 \ldots s_m$ as $t_1 \ldots t_m$ for some $t_i \in T$. Hence $h \in t_1 \ldots t_m K$. For the second inequality, by \cref{generating_set_for_the_subgroup_in_terms_of_ball}, $S^{3 \dgk  }\cap H$ generates $H$. Therefore, ($S^{3 \dgh }\cap H )^{\dgk}$ contains a set of coset representative for $K$ in $H$. Hence $H \subeq S ^{3\dgh  \dgk } K$. Hence, $$G = S ^{\diam(G/H)} H =  S ^{\diam(G/H)} (S ^{3\dgh \dgk } K) = S ^{4\dgh \dgk } K.$$
\end{proof}

\begin{lem} \label{uniform_D_alpha_passes_finte_index_subgroup}
Suppose $ G = \langle S \rangle$ has uniformly   $\alpha$-almost flat coset spaces and $H$ is a finite index subgroup of $G$, then $H$ has uniformly   $\alpha$-almost flat coset spaces.
\end{lem}

\begin{proof}
Let $H$ be a subgroup of $G$ with index $k$.  According to \cref{generating_set_for_the_subgroup_in_terms_of_ball}, $H$ is generated by the finite subset $T \coloneqq H \cap S^{3k}$. Let $K$ be a finite index subgroup of $H$, since $G$ has uniformly   $\alpha$-almost flat coset spaces, there exists $c>0$ such that $\diam_S(G/K) \geq c |G:K|^\a$. By \cref{boundind_diam_G_K}, we have  
\[ \diam_T (H/K) \geq \frac{\dgk}{4\dgh} \geq \frac{c|G:K|^\a}{4|G:H|} = \frac{c|H:K|^\a}{4|G:H|^{1-\a}}.\qedhere\] 
\end{proof}

It turns out that both having uniformly    almost flat coset spaces and  uniformly    almost flat quotients pass from the finite-index subgroup to the whole group. 

\begin{lem} \label{uniform_D_alpha_passes_to_the_whole_group}
Suppose $G =\langle S \rangle$ has a finite-index subgroup $H$ which has uniformly   $\alpha$-almost flat   coset (resp. quotients) spaces, then $G$ has uniformly   $\alpha$-almost flat coset (resp. quotients) spaces.
\end{lem}

\begin{proof} 
Let $T$ be the generating set for $H$ defined in \cref{generating_set_for_the_subgroup_in_terms_of_ball}. 
Suppose $L$ is a finite index subgroup (resp. normal subgroup) of $G$, then $K \coloneqq H \cap L$ is a finite index subgroup (resp. finite index normal subgroup) of $H$ with $|L:K|\leq|G:H|$. 
 Therefore, \cref{boundind_diam_G_K} tells us that $$\diam_S(G/K) \geq \diam_T (H/K) \geq  c|H:K|^\a = c \frac{|G:L|^\a |L: K|^\a } { |G:H |^\a}  \geq c \frac{|G:L|^\a} { |G:H |^\a}.$$
Applying  \cref{boundind_diam_G_K} again, we can deduce that 
 \[ \diam_S (G/L) \geq \frac{ \dgk }{4 |L:K|} \geq \frac{ c |G:L|^\a }{4 |G:H| |G:H|^\a} = \frac{ c |G:L|^\a }{4  |G:H|^{1+\a}}. \qedhere \]  
\end{proof}

\section{Example: The Heisenberg group}
In this short section we prove \cref{prop:heis}.
\begin{proof}[Proof of \cref{prop:heis}]
Write $H$ for the Heisenberg group, let $S_0$ be the generating set
\[
x=\left(
\begin{array}{ccc}
1&1&0\\
0&1&0\\
0&0&1\\
\end{array}
\right),\quad
y=\left(
\begin{array}{ccc}
1&0&0\\
0&1&1\\
0&0&1\\
\end{array}
\right),\quad
z=\left(
\begin{array}{ccc}
1&0&1\\
0&1&0\\
0&0&1\\
\end{array}
\right),
\]
and set $S=S_0\cup S_0^{-1}\cup\{1\}$. Note that $z=[x,y]$ and that $[H,H]=Z(H)=\la z\ra$. Note moreover that the commutator is bilinear in the sense that $[x^k,y^\ell]=z^{k\ell}$ for all $k,\ell\in\Z$.

To see that $H$ has uniformly $(1/3)$-almost flat quotients, suppose $G$ is a finite quotient of $H$, with quotient homomorphism $f:H\to G$. The abelianisation of $G$ is a quotient of $H/[H,H]\cong\Z^2$, so by \cref{lem:poly.growth} we have $|G/[G,G]|\ll\diam_S(G/[G,G])^2\le\diam_S(G)^2$. Moreover, setting $m=\min\{n\in\N:f(x^n)\in\la f(y),f(z)\ra\}$, we certainly have $\diam_S(G)\ge m/2$. Centrality of $z$ implies that $[f(x^m),f(y)]=1 $, and bilinearity of the commutator therefore implies that $|[G,G]|\le m\le\diam_S(G)$, so that $|G|\ll\diam_S(G)^3$ as required.

To see that $H$ does not have uniformly $\alpha$-almost flat quotients for any $\alpha>1/3$, for each prime $p$ let $K_p$ be the normal subgroup
\[
K_p=\left(
\begin{array}{ccc}
1&p\Z&p\Z\\
0&1&p\Z\\
0&0&1\\
\end{array}
\right).
\]
A complete set of coset representatives for $K_p$ in $H$ is given by
\[
z^cy^bx^a=\left(
\begin{array}{ccc}
1&a&c\\
0&1&b\\
0&0&1\\
\end{array}
\right),\qquad0\le a,b,c<p,
\]
so that $[H:K_p]=p^3$ and $\diam_S(H/K_p)\le3p$.

It is well known and easy to check that $H$ has polynomial growth of degree $4$, so \cref{lem:poly.growth} implies that it has uniformly $(1/4)$-almost flat coset spaces. To see that $H$ does not have uniformly $\beta$-almost flat coset spaces for any $\beta>1/4$, for each $n\in\N$ let $H_n$ be the subgroup
\[
H_n=\left(
\begin{array}{ccc}
1&n\Z&n^2\Z\\
0&1&n\Z\\
0&0&1\\
\end{array}
\right).
\]
Similarly to $K_p$ a complete set of coset representatives for $H_n$ in $H$ is given by
\[
z^cy^bx^a=\left(
\begin{array}{ccc}
1&a&c\\
0&1&b\\
0&0&1\\
\end{array}
\right),\qquad0\le a,b<n,\quad0\le c<n^2,
\]
so that $[H:H_n]=n^4$. To obtain the desired result, we will show moreover that $\diam_S(H/H_n)\le10n$. To see this, for any $c\in\{0,\ldots,n^2-1\}$, write $c$ in base $n$, that is to say write $c=c_0+c_1n$ with $c_0,c_1\in\{0,1,\ldots,n-1\}$, and then note that $z^c=[x^{c_0},y][x^{c_1},y^n]\in S^{8n}$. An arbitrary coset representative $z^cy^bx^a$ therefore belongs to $S^{10n}$, and the claim is proved.
\end{proof}

\section{Khukhro and Valette's results}\label{sec:kv}
In this section we prove Khukhro and Valette's results on residual almost flatness in the form we stated in the introduction. In proving \cref{thm:kv.detailed} we make use of the standard fact that a finitely generated group has finitely many subgroups of any given finite index.
\begin{lem}[{\citep[Corollary 1.1.2]{lub-seg}}]\label{lem:finitely.many.subgroups.index}
Suppose $G$ is a group generated by at most $r$ elements, and let $n\in\N$. Then $G$ has at most $n(n!)^{r-1}$ subgroups of index $n$.
\end{lem}

\begin{proof}[Proof of \cref{thm:kv.detailed}]
The implication \eqref{item:box}$\implies$\eqref{item:raf} is trivial. The implication \eqref{item:indic}$\implies$\eqref{item:box} is proved in \citep[Theorem~4]{box}; indeed, it is shown there that for any finite generating set $S$ there exists a decreasing sequence $(N_k)_{k\in\N}$ of finite-index normal subgroups with trivial intersection and $\eps,\alpha\in(0,1)$ such that $\diam_S(G/N_k)\ge\eps[G:N_k]^\alpha$ for all $k$. It remains to prove \eqref{item:raf}$\implies$\eqref{item:indic}, which we do following the proof of \citep[Corollary~1.7]{nil}.

Fix a finite generating set $S$. By definition, if \eqref{item:raf} holds then there exist $\eps,\alpha\in(0,1)$ and, for each $g\in G$, a normal subgroup $N_g\normal G$ of finite index such that $g\notin N_g$ and $\diam_S(G/N_g)\ge\eps[G:N_g]^\alpha$. As $g$ ranges over all the elements of $G$, the indices $|G:N_g|$ must be unbounded, since if this were not the case then \cref{lem:finitely.many.subgroups.index} would imply that $\{N_g:g\in G\}$ was a finite set, and then $\bigcap_{g\in G}N_g$ would contain non-identity elements. 

Fix $d\in\N$ so that $2/(d+2)<\alpha$, and let $A=A(d)>0$ be as in \cref{thm:bt}. Since $[G:N_g]$ is unbounded, there exist $g\in G$ with arbitrarily large $|G:N_g|$ such that $\diam_S(G/N_g)\ge A|G:N_g|^{2/(d+2)}$. \cref{thm:bt} then implies that for each of these $g$ there exist a subgroups $H_g\normal G$ and $K_g\le G$ with $H_g\le K_g$ and $|G:K_g|\le(2d)!$ such that $H_g\subseteq S^{O_d(\diam_S(G/N_g)^{1/2})}N_g$ and $K_g/H_g$ is abelian. 
Note that since $|G:N_g|$ is unbounded over $g\in G$, so is $\diam_S(G/N_g)$, hence so is $\diam_S(G/H_g)$, and hence so is $|G:H_g|$. Since $|G:K_g|\le(2d)!$, this in turn implies that $|K_g:H_g|$ is unbounded over $g\in G$. However, \cref{lem:finitely.many.subgroups.index} implies that the set $\{K_g:g\in G\}$ is finite, so there exists $g_0\in G$ such that $K_{g_0}$ admits abelian quotients of unbounded size, and hence has infinite abelianisation. Since $K_{g_0}$ is of finite index in $G$ and hence finitely generated, it follows that $K_{g_0}$ surjects onto $\Z$.
\end{proof}

\begin{proof}[Proof of \cref{thm:kv.1}]
The implications \eqref{item:1afcs} $\implies$ \eqref{item:1afq} $\implies$ \eqref{item:r1af} are trivial, and since virtually cyclic groups have linear growth the implication \eqref{item:vz} $\implies$ \eqref{item:1afcs} follows from \cref{lem:poly.growth}. It remains to prove that \eqref{item:r1af} $\implies$ \eqref{item:vz}, which we do following Khukhro and Valette.

Suppose then that $G$ is strongly residually $1$-almost flat, and fix a finite symmetric generating set $S$ containing the identity. In particular, this means that there exists $\eps>0$ such that for all $r\in\N$ there exists a finite-index normal subgroup $N_r\normal G$ such that $\diam_S(G/N_r)\ge\eps[G:N_r]$ and $N_r\cap S^{2r}=\{1\}$. Note that this last condition implies that $S^r$ injects into $G/N_r$, so that the ball of radius $r$ in $G/N_r$ has cardinality $|S^r|$.

Fix $r$, and pick $x,y\in G/N_r$ at distance $\diam_S(G/N_r)$. Set $k=\lfloor\diam_S(G/N_r)/(2r+1)\rfloor$, and let $x_0,\ldots,x_k$ lie on a geodesic between $x$ and $y$ with $d(x_i,x_{i+1})=2r+1$. Then the balls of radius $r$ centred at the points $x_i$ are disjoint, so that
\[
(k+1)|S^r|\le[G:N_r]\le\eps^{-1}\diam_S(G/N_r)\le\eps^{-1}(k+1)(2r+1)
\]
and hence $|S^r|\le\eps^{-1}(2r+1)$. It follows that $G$ has linear growth, and hence is virtually cyclic by \citep{wilkie-vdd}.
\end{proof}

\section{Residually nilpotent groups}
In this section we prove \cref{thm:resid,thm:nilp,thm:nilp.coset}. We start by recalling some necessary background on nilpotent groups. Given elements $x,y$ of a group $G$ we define the \emph{commutator} $[x,y]=x^{-1}y^{-1}xy$. Given elements $x_1,\ldots,x_k\in G$ we then define the \emph{simple commutator} $[x_1,\ldots,x_k]$ recursively via $[x_1,\ldots,x_k]=[[x_1,\ldots,x_{k-1}],x_k]$. We call $k$ the \emph{weight} of this simple commutator. We extend these notions to subgroups $H_1,\ldots,H_k$ by defining $[H_1,H_2]=\langle[h_1,h_2]:h_i\in H_i\rangle$ and then recursively $[H_1,\ldots,H_k]=[[H_1,\ldots,H_{k-1},H_k]$.

We define the \emph{lower central series} $G_1\ge G_2\ge\cdots$ of $G$ by $G_1=G$ and $G_{i+1}=[G,G_i]$, noting in particular that each quotient $G_i/G_{i+1}$ is abelian, in fact central in $G/G_{i+1}$. The group $G$ is \emph{nilpotent} if $G_{k+1}=\{e\}$ for some $k\in\N$; the smallest such $k$ is called the \emph{class} of $G$.

It is well known that
\begin{equation}\label{eq:[G_i,G_j]}
[G_i,G_j]\subseteq G_{i+j}
\end{equation}
for all $i,j\in\N$. It is also well known that $G_i/G_{i+1}$ is generated by simple commutators of weight $i$.

\begin{lem}[{\citep[Lemmas 5.5.2 \& 5.5.3]{tointon.book}}]\label{lem:multilin}
Let $G$ be a group, and let $i,j,k\in\N$. Then the maps
\[
\begin{array}{ccc}
G_i\times G_j&\to&G_{i+j}/G_{i+j+1}\\
(x,y)&\mapsto&[x,y]G_{i+j+1}
\end{array}
\]
and
\[
\begin{array}{ccc}
G\times\cdots\times G&\to&G_k/G_{k+1}\\
(x_1,\ldots,x_k)&\mapsto&[x_1,\ldots,x_k]G_{k+1}
\end{array}
\]
are homomorphisms in each variable.
\end{lem}

The set of finite-order elements of a nilpotent group $G$ form a subgroup $T(G)\le G$, called the \emph{torsion subgroup} of $G$. The torsion subgroup is trivially characteristic, and $G/T(G)$ is torsion-free. If $G$ is finitely generated then $T(G)$ is finite.

Following \citep{mpty}, for each $k\in\N$ we also define the \emph{generalised commutator subgroup} $\overline G_k$ by $\overline G_k=\{g\in G:(\exists n\in\N)(g^n\in G_k)\}$. The generalised commutator subgroups are all characteristic in $G$, each $\overline G_i/\overline G_{i+1}$ is torsion-free abelian, and if $G$ is finitely generated then each $\overline G_k$ contains $G_k$ as a finite-index subgroup \citep[Lemma 4.1]{mpty}. Moreover, we have $[\overline G_i,\overline G_j]\le\overline G_{i+j}$ for all $i,j\in\N$ \citep[Lemma 4.5]{mpty}.

If $N$ is a simply connected nilpotent Lie group with Lie algebra $\nn$ then there exist mutually inverse diffeomorphisms $\exp:\nn\to N$ and $\log:N\to\nn$ 
\citep[Theorem 1.2.1]{cor-gre}. 
The lower central series $\nn=\nn_1\ge\nn_2\ge\cdots$ of $\nn$ is defined recursively by $\nn_{i+1}=[\nn,\nn_i]=\Span_\R\{[x,y]:x\in\nn,y\in\nn_i\}$. The terms $N_i$ of the lower central series of $N$ are Lie subgroups and $\exp N_i=\nn_i$ for each $i$ \citep[Theorem XII.3.1]{hochschild}. 
Simply connected nilpotent Lie groups are \emph{uniquely divisible}: for every $x\in N$ and $k\in\N$ there exists a unique $y\in N$ such that $y^k=x$.

The \emph{structure constants} of a basis $v_1,\ldots,v_h$ of a Lie algebra $\nn$ are the unique $c_{ijk}\in\R$ such that
\[
[v_i,v_j]=\sum_{k=1}^hc_{ijk}v_k
\]
for all $i,j$.

Every finitely generated torsion-free nilpotent group $G$ embeds as a discrete cocompact subgroup in a simply connected nilpotent Lie group $N$, unique up to isomorphism, called the \emph{Mal'cev completion} of $G$ \citep{malcev,raghu}. Moreover, there exists a certain basis for the Lie algebra of $N$ that is compatible with $G$ in a rather strong sense, as follows.
\begin{thm}\label{thm:malcev.basis}
Suppose $N$ is a simply connected nilpotent Lie group with Lie algebra $\nn$ with dimension $h$, and $G$ is a discrete cocompact subgroup. Then there exists a basis $v_1,\cdots,v_h$ of $\nn$ with rational structure constants such that, writing $x_i=\exp v_i$ for each $i$, the following conditions are satisfied.
\begin{enumerate}[(i)]
\item Every element $g\in N$ can be written uniquely as $g=x_1^{\ell_1}\cdots x_h^{\ell_h}$ with $\ell_i\in\R$.\label{item:malcev.unique}
\item Writing $h_i=\dim\nn-\dim\nn_{i+1}$ for each $i$, we have $N_i=\{x_{h_{i-1}+1}^{\ell_{h_{i-1}+1}}\cdots x_h^{\ell_h}:\ell_i\in\R\}$. \label{item:basis.pass.lcs}
\item We have $G=\{x_1^{\ell_1}\cdots x_h^{\ell_h}:\ell_i\in\Z\}$.\label{item:malcev.lattice}
\end{enumerate}
\end{thm}
\begin{proof}
This is standard and implicit in \citep{cor-gre}, so we provide only brief details of how it may be read out of that source. The normal subgroups $N_i$ are rational by \citep[Theorem 5.1.8 \& Corollary 5.2.2]{cor-gre}, so the existence of a basis satisfying conditions \eqref{item:malcev.unique}--\eqref{item:malcev.lattice} follows from \citep[Corollary 5.3.2 \& Proposition 1.2.7]{cor-gre}. The proof of \citep[Theorem 5.1.8 (a)]{cor-gre} then shows (as stated in the second sentence of the proof) that the structure constants are rational.
\end{proof}
A basis satisfying conditions \eqref{item:malcev.unique} and \eqref{item:basis.pass.lcs} of \cref{thm:malcev.basis} is called a \emph{strong Mal'cev basis for $\nn$ passing through the lower central series}; if such a basis also satisfies \eqref{item:malcev.lattice} then it is said to be \emph{strongly based on $G$}.

In the following proposition and elsewhere, given a group $G$ and $m\in\N$ we write $G^m=\la g^m:g\in G\ra$.
\begin{prop}\label{prop:tfn.reduc.mod.p}
Suppose $G$ is a torsion-free nilpotent group with Hirsch length $h$. Then $G$ has a generating set $x_1,\ldots,x_h$ such that
\begin{enumerate}[(i)]
\item every element $g\in G$ can be written uniquely as $g=x_1^{\ell_1}\cdots x_h^{\ell_h}$ with $\ell_i\in\Z$;\label{item:unique}
\item $[x_i,x_j]\in\langle x_{j+1},\ldots,x_h\rangle$ whenever $1\le i<j\le h$;\label{item:triangular}
\item for all but finitely many primes $p$ we have $G^p=\{x_1^{p\ell_1}\cdots x_h^{p\ell_h}:\ell_i\in\Z\}$.\label{item:G^p}
\end{enumerate}
\end{prop}
\begin{remark*}
Conversely, the existence of a generating set satisfying \eqref{item:unique} and \eqref{item:triangular} implies that $G$ is torsion-free nilpotent.
\end{remark*}
\begin{proof}
A result of Jennings \citep{jennings} states that $G$ can be embedded in an upper unitriangular integer matrix group $\Gamma$. We will prove that $\Gamma$ satisfies the proposition, and also that the property of having a generating set satisfying \eqref{item:unique}--\eqref{item:G^p} is inherited by subgroups.

There exist elementary matrices $x_1,\ldots,x_d\in\Gamma$ such that every element $g\in\Gamma$ can be written uniquely as $g=x_1^{\ell_1}\cdots x_d^{\ell_d}$ with $\ell_i\in\Z$ exactly the non-diagonal matrix entries of $g$, and such that $[x_i,x_j]\in\langle x_{j+1},\ldots,x_d\rangle$ whenever $1\le i<j\le d$. Write $K_p=\{x_1^{p\ell_1}\cdots x_d^{p\ell_d}:\ell_i\in\Z\}$. Note $K_p$ consists of exactly of those matrices of the form $I+pB$, with $B$ a strictly upper-triangular matrix. Note that if $g=I+pB$ and $h=I+pC$ are two such matrices then $gh=I+p(B+C+pBC)\in K_p$, and also that $x_i^{-p}\in K_p$ for each $i$, so that $K_p$ is a subgroup for all primes $p$.

Clearly $K_p\le\Gamma^p$ for all primes $p$. Conversely, let $g\in\Gamma$, and write $g=I+A$, with $I$ the identity matrix and $A$ a strictly upper-triangular matrix. Writing $n$ for the degree of $\Gamma$, we have $A^n=0$, and hence for $p\ge n$ the binomial theorem gives $g^p=I + \sum_{m=1}^{n-1}{p \choose m}A^m=I+pB$ for some strictly upper-triangular matrix $B$, so that $g\in K_p$. Thus $\Gamma$ satisfies the proposition.

Now suppose that $\Gamma$ is an arbitrary group with Hirsch length $d$, and that $x_1,\ldots,x_d$ is generating set for $\Gamma$ satisfying \eqref{item:unique}--\eqref{item:G^p}. We will prove by induction on $d$ that every subgroup of $\Gamma$ satisfies the proposition. Suppose, then, that $G\le\Gamma$ has Hirsch length $h$. Let $p$ be a prime such that the generating set $x_i$ satisfies \eqref{item:G^p}, noting that
\begin{equation}\label{eq:G^p.1}
G^p\subseteq\{x_1^{p\ell_1}\cdots x_d^{p\ell_d}:\ell_i\in\Z\}.
\end{equation}
Since $G/G\cap\la x_d\ra$ is isomorphic to a subgroup of $\Gamma/\la x_d\ra$, we may assume by induction on $d$ that $G/G\cap\la x_d\ra$ satisfies the proposition. If $G\cap\la x_d\ra\ne\{e\}$ then $G$ itself therefore satisfies the proposition, so we may assume that $G\cap\la x_d\ra\ne\{e\}$, hence $G\cap\la x_d\ra=\la x_d^m\ra$ for some $m\in\N$. This implies that $G/G\cap\la x_d\ra$ has Hirsch length $h-1$, so that there exists a generating set $z_1\la x_d^m\ra,\ldots,z_{h-1}\la x_d^m\ra$ for $G/G\cap\la x_d\ra$ satisfying \eqref{item:unique}--\eqref{item:G^p}. In particular, this means that
\[
(G/G\cap\la x_d\ra)^p=\{z_1^{pk_1}\cdots z_{h-1}^{pk_{h-1}}\la x_d^m\ra:k_i\in\Z\},
\]
and hence
\begin{equation}\label{eq:G^p.2}
G^p\subseteq\{z_1^{pk_1}\cdots z_{h-1}^{pk_{h-1}}x_d^{mk_h}:k_i\in\Z\}.
\end{equation}
Now for each choice of $k_i\in\Z$ we have $z_1^{pk_1}\cdots z_{h-1}^{pk_{h-1}}\in\{x_1^{p\ell_1}\cdots x_d^{p\ell_d}:\ell_i\in\Z\}$, so \eqref{eq:G^p.1} and \eqref{eq:G^p.2} combine to imply that $G^p\subseteq\{z_1^{pk_1}\cdots z_{h-1}^{pk_{h-1}}x_d^{mk_h}:k_i\in\Z,\,p\mid mk_h\}$. Provided $p$ does not divide $m$, it follows that
\[
G^p=\{z_1^{p\ell_1}\cdots z_{h-1}^{p\ell_{h-1}}x_d^{mp\ell_h}:\ell_i\in\Z\},
\]
and so the required generating set is $z_1,\ldots,z_{h-1},x_d^m$.
\end{proof}

\begin{prop}\label{prop:tfn.M^n.subgrp}
Suppose $G$ is a torsion-free nilpotent group with class $c$ and Hirsch length $h$. For each $i=1,\ldots,c-1$, write $a_i$ for the Hirsch length of $G_i/G_{i+1}$, and define $b(j)=1$ for $j=1,\ldots,a_1$ and $b(j)=2$ for $j=a+1,\ldots,h$. Then $G$ has a generating set $x_1,\ldots,x_h$ such that
\begin{enumerate}[(i)]
\item every element $g\in G$ can be written uniquely as $g=x_1^{\ell_1}\cdots x_h^{\ell_h}$ with $\ell_i\in\Z$;\label{item:unique.M^n}
\item $[x_i,x_j]\in\langle x_{j+1},\ldots,x_h\rangle$ whenever $1\le i<j\le h$;\label{item:triangular.M^n}
\item there exists $q\in\N$ such that for each $m\in\N$ the set $\{x_1^{qm^{b(1)}\ell_1}\cdots x_h^{qm^{b(h)}\ell_h}:\ell_i\in\Z\}$ is a subgroup of $G$.\label{item:M^n.subgroup}
\end{enumerate}
\end{prop}
\begin{proof}
Write $N$ for the Mal'cev completion of $G$ and $\nn$ for the Lie algebra of $N$. Let $e_1,\ldots,e_h$ be a strong Mal'cev basis for $\nn$ passing through the lower central series strongly based on $G$, and write $c_{ijk}\in\Q$ for the structure constants. Setting $x_i=\exp e_i$ for each $i$, we obtain a generating set satisfying \eqref{item:unique.M^n} and \eqref{item:triangular.M^n}.

Let $d$ be the lowest common denominator of the rationals $c_{ijk}$. For each $i$ set $f_i=de_i$, noting that $[f_i,f_j]=\sum_{k=a+1}^hdc_{ijk}f_k$ for all $i,j$. Let us emphasise that the coefficients $dc_{ijk}$ in this sum are integers. Given $m\in\N$, we therefore have $[m^{b(i)}f_i,m^{b(j)}f_j]=\sum_{k=a+1}^hdc_{ijk}m^{b(i)+b(j)}f_k$ for all $i,j$. Writing $\Lambda$ for the additive subgroup of $\nn$ spanned by $m^{b(i)}f_i$, it follows that $[\Lambda,\Lambda]\subseteq\Lambda$, so \citep[Lemma 4.3]{proper.progs} implies that there exists $r$ depending only on $c$ such that $\exp(r\Lambda)$ is a subgroup of $N$, and \citep[Corollary 3.11]{ttBalls} then implies that $\exp(r\Lambda)=\{x_1^{drm^{b(1)}\ell_1}\cdots x_h^{drm^{b(h)}\ell_h}:\ell_i\in\Z\}$, giving \eqref{item:M^n.subgroup}.
\end{proof}

\begin{lem}\label{lem:G^p.cosets}
Suppose $G$ is a torsion-free nilpotent group with Hirsch length $h$, and $x_1,\ldots,x_h$ is a generating set such that every element $g\in G$ can be written uniquely as $g=x_1^{\ell_1}\cdots x_h^{\ell_h}$ with $\ell_i\in\Z$, and $[x_i,x_j]\in\langle x_{j+1},\ldots,x_h\rangle$ whenever $1\le i<j\le h$. Suppose further that $H=\{x_1^{q_1\ell_1}\cdots x_h^{q_h\ell_h}:\ell_i\in\Z\}$ is a subgroup of $G$ for some positive integers $q_1,\ldots,q_h$. Then $X=\{x_1^{k_1}\cdots x_h^{k_h}:0\le k_i<q_i\}$ contains exactly one element from each coset of $H$ in $G$.
\end{lem}
\begin{proof}
Let $\ell_1,\ldots,\ell_h\in\Z$. We may assume by induction on $d$ that the lemma holds modulo $\langle x_h\rangle$, so that there exist $n_1,\ldots,n_{h-1},r_1,\ldots,r_{h-1}\in\Z$ with $0\le r_i<q_i$, as well as some $m\in\Z$, such that
\[
x_1^{\ell_1}\cdots x_{h-1}^{\ell_{h-1}}=x_1^{r_1}\cdots x_{h-1}^{r_{h-1}}x_1^{q_1n_1}\cdots x_{h-1}^{q_{h-1}n_{h-1}}x_h^m
\]
and hence
\[
x_1^{\ell_1}\cdots x_h^{\ell_h}=x_1^{r_1}\cdots x_{h-1}^{r_{h-1}}x_1^{q_1n_1}\cdots x_{h-1}^{q_{h-1}n_{h-1}}x_h^{m+\ell_h}.
\]
Letting $n_h,r_h\in\Z$ be such that $m+\ell_h=q_hn_h+r_h$ and $0\le r_h<q_h$, by centrality of $x_h$ we then have
\[
x_1^{\ell_1}\cdots x_h^{\ell_h}=x_1^{r_1}\cdots x_h^{r_h}x_1^{q_1n_1}\cdots x_h^{q_hn_h},
\]
so that $X$ contains a representative of the coset $x_1^{\ell_1}\cdots x_h^{\ell_h}$. To see that there are no duplicates, suppose $(k_1,\ldots,k_h)$ and $(\ell_1,\ldots,\ell_h)$ are distinct tuples with $0\le k_i,\ell_i<q_i$, and let $a=x_1^{k_1}\cdots x_h^{k_h}$ and $b=x_1^{\ell_1}\cdots x_h^{\ell_h}$. Let $i$ be minimal such that $k_i\ne\ell_i$. Then $a^{-1}b=x_i^{\ell_i-k_i}\cdots\notin H$, so $a$ and $b$ belong to distinct cosets.
\end{proof}

\begin{proof}[Proof of \cref{thm:nilp,thm:nilp.coset}]
We may assume that $N$ is a normal torsion-free nilpotent subgroup of finite index in $G$. Write $c$ for the class of $N$. For each $i=1,\ldots,c-1$ write $a_i$ for the Hirsch length of $G_i/G_{i+1}$, and write $h_i=a_1+\cdots+a_i$.

Let $x_1,\ldots,x_h$ be the generating set for $N$ given by \cref{prop:tfn.reduc.mod.p}, let $1=y_1,\ldots,y_k$ be a set of coset representatives for $N$ in $G$, and let $S_0=\{x_1,\ldots,x_h,y_1,\ldots,y_k\}$ and $S=S_0\cup S_0^{-1}$. Given a prime $p$, the subgroup $N^p$ is characteristic in $N$ and hence normal in $G$. Moreover, for the all but finitely many primes for which property \eqref{item:G^p} of \cref{prop:tfn.reduc.mod.p} holds for the generating set $x_1,\ldots,x_h$, \cref{lem:G^p.cosets} implies that $[G:N_p]=kp^h$. On the other hand, \cref{lem:G^p.cosets} also implies that $\diam_S(G/N^p)\le1+hp$ for all such primes. Letting $p\to\infty$ therefore shows that $G$ does not have uniformly $\alpha$-almost flat quotients for any $\alpha>1/h$.

To show that $G$ does not have uniformly $\beta$-almost flat coset spaces for $\beta>1/(h+h')$, it suffices by \cref{uniform_D_alpha_passes_finte_index_subgroup} to show that $N$ does not have uniformly $\beta$-almost flat coset spaces for $\beta>1/(h+h')$. To see this, let $S_0=\{z_1,\ldots,z_h\}$ be the generating set for $N$ given by \cref{prop:tfn.M^n.subgrp}, and let $S=S_0\cup S_0^{-1}\cup\{1\}$. By definition of $S_0$ there exists $q\in\N$ such that for each $m\in\N$ the set $H_m=\{z_1^{qm^{b(1)}\ell_1}\cdots z_h^{qm^{b(h)}\ell_h}:\ell_i\in\Z\}$ is a subgroup of $N$, where $b(i)=1$ for $i\le h-h'=a_1$ and $b(i)=2$ otherwise. It then follows from \cref{lem:G^p.cosets} that $[N:H_m]=q^hm^{h+h'}$.

Since $N_i/\overline N_{i+1}$ is generated by the simple commutators of weight $i$ in the elements of $S_0$, there exists a subset $u_{i1},\ldots,u_{ia_i}$ of such simple commutators that generates a finite-index subgroup $K_i/\overline N_{i+1}$ of $\overline N_i/\overline N_{i+1}$. Given $m\in\N$, we claim that for each $j$ we have
\begin{equation}\label{eq:simple.coms.m^i}
u_{ij}^r\in S^{O_i(m)}N_{i+1}\qquad\qquad(r=1,\ldots,m^i).
\end{equation}
First, writing $r$ in base $m$ we obtain $b\le i$ and integers $\ell_{11},\ldots,\ell_{1i},\ldots,\ell_{b1},\ldots,\ell_{bi}\in[1,m]$
such that 
$r=\sum_{j=1}^b\prod_{t=1}^i\ell_{jt}$. \cref{lem:multilin} therefore implies that for any simple commutator $[s_1,\ldots,s_i]$ of weight $i$ in the elements of $S_0$ we have
\[
[s_1,\ldots,s_i]^r\in[s_1^{\ell_{11}},\ldots,s_i^{\ell_{1i}}]\cdots[s_1^{\ell_{b1}},\ldots,s_i^{\ell_{bi}}]N_{i+1}\subseteq S^{O_i(m)}N_{i+1}
\]
as claimed.

We now claim that $K_i\subseteq S^{O_i(a_iqm)}H_m\overline N_{i+1}$. Indeed, since $K_i/\overline N_{i+1}$ is abelian, the subgroup $(H_m\cap K_i)\overline N_{i+1}$ is normal in $K_i$. The quotient $K_i/(H_m\cap K_i)\overline N_{i+1}$ is then an abelian group with exponent dividing $qm^{b(i)}$, and in particular dividing $qm^i$, and it then follows from \eqref{eq:simple.coms.m^i} that in fact $K_i\subseteq S^{O_i(a_iqm)}(H_m\cap K_i)\overline N_{i+1}$. Since $K_i$ has finite index in $\overline N_i$, it follows that there is some constant $b_i$, independent of $m$, such that $\overline N_i\subseteq S^{b_im}H_m\overline N_{i+1}$, and then writing $b=b_1+\cdots+b_c$ we have $N\subseteq S^{b_im}H_m$. In particular, $\diam_{S}(N/H_m)\ll m$ as $m\to\infty$, and $N$ does not have uniformly $\beta$-almost flat coset spaces for $\beta>1/(h+h')$.
\end{proof}

\begin{lem}\label{lem:resid.tfn}
Let $G$ be a group. Then $G$ is residually torsion-free nilpotent if and only if $\bigcap_{i=1}^\infty\overline G_i=\{e\}$.
\end{lem}
\begin{proof}
Since each $G/\overline G_i$ is torsion-free nilpotent, $G$ is certainly residually torsion-free nilpotent if $\bigcap_{i=1}^\infty\overline G_i=\{e\}$. Conversely, if $N\normal G$ is such that $G/N$ is torsion-free nilpotent of class $i$ then $N\supseteq\overline G_{i+1}$, so $\bigcap_{i=1}^\infty\overline G_i$ is exactly the set of elements that have trivial image in all torsion-free nilpotent quotients of $G$.
\end{proof}

\begin{lem}\label{lem:lastG_k}
Let $G$ be a group, let $k\in\N$, and suppose $[\overline G_k:\overline G_{k+1}]<\infty$. Then $\bigcap_{i=1}^\infty\overline G_i=\overline G_k$.
\end{lem}
\begin{proof}
We will show that if $[\overline G_k:\overline G_{k+1}]<\infty$ then $\overline G_k=\overline G_{k+2}$. By induction this will imply that $\overline G_k=\overline G_\ell$ for all $\ell\ge k$, and hence prove the lemma. The quotient group $\overline G_k/\overline G_{k+1}$ is torsion-free abelian, so in this case it must be trivial, so that $\overline G_k=\overline G_{k+1}$. This implies in particular that $G_k\subseteq\overline G_{k+1}$, and hence that $G_k/G_{k+1}$ is a torsion group. It follows from \cref{lem:multilin} that the map
\[
\begin{array}{ccc}
G\times G_k&\to&G_{k+1}/G_{k+2}\\
(x,y)&\mapsto&[x,y]G_{k+2}
\end{array}
\]
is a homomorphism in each variable, and from \eqref{eq:[G_i,G_j]} that $G_{k+1}$ is in the kernel with respect to the second variable. Since $G_k/G_{k+1}$ is a torsion group, this implies that $G_{k+1}/G_{k+2}$ is generated by elements of finite order. Since $G_{k+1}/G_{k+2}$ is abelian, this implies that $G_{k+1}/G_{k+2}$ is a torsion group, hence that $G_{k+1}\subseteq\overline G_{k+2}$, and hence that $\overline G_{k+1}=\overline G_{k+2}$.
\end{proof}

\begin{proof}[Proof of \cref{thm:resid}]
\cref{lem:resid.tfn,lem:lastG_k} imply that if $\overline G_k=\overline G_{k+1}$ for some $k\in\N$ then $\overline G_k=\{e\}$. In particular, this implies that if the Hirsch length of $G/G_k$ equals that of $G/G_{k+1}$ then $\overline G_k=\{e\}$. 

We will firstly show that if \( G \) is not nilpotent with Hirsch length \(\leq d\), then \( G \) does not have uniformly \((1/d)\)-almost flat quotient spaces. Suppose we also have \( h(G/G_{d+2}) \leq d \).  Then we have the following inequalities
\[
0 = h(G/G_1) \leq h(G/G_2) \leq \cdots \leq h(G/G_{d+2}) \leq d.
\]
In particular, there exists some \( i \leq d+1 \) such that \( h(G/G_i) = h(G/G_{i+1}) \). It follows that \( \overline{G}_i = \{ e \} \) and \( h(G) = h(G/G_i) \leq d \), so \( G \) is nilpotent with Hirsch length \(\leq d\), contradicting our assumption. Hence, \( h(G/G_{d+2}) > d \). The finite quotients of \( G/G_{d+2} \) are all also finite quotients of \( G \), and by \cref{thm:nilp} they are not uniformly \((1/d)\)-almost flat. 

Similarly, suppose \( G \) is not nilpotent and that \( h(G) + h([G, G]) \leq d \). Suppose in addition that
\[
h(G/G_{d+2}) + h(G_2/G_{d+2}) \leq d.
\]
Then we have the following chain of inequalities:
\[
\begin{split}0 \leq h(G/G_1) \leq h(G/G_2) + h(G_2/G_2) \leq h(G/G_3) + h(G_2/G_3) \leq \cdots\qquad\qquad \\\leq h(G/G_{d+2}) + h(G_2/G_{d+2}) \leq d.\end{split}
\]

If \( h(G/G_1) = h(G/G_2) \), then \( \overline{G}_1 = \{ e \} \), so \( G \) is the trivial group, which clearly contradicts our assumption. Otherwise, there exists some \( i \leq d+1 \) such that
\[
h(G/G_i) + h(G_2/G_i) = h(G/G_{i+1}) + h(G_2/G_{i+1}).
\]
It follows that \( \overline{G}_i = \{ e \} \), and
\[
h(G) + h([G, G]) = h(G/G_i) + h(G_2/G_i) \leq d,
\]
so \( G \) is nilpotent with \( h(G) + h([G, G]) \leq d \), again contradicting our assumption. Hence,
\[
h(G/G_{d+2}) + h(G_2/G_{d+2}) > d.
\]
The finite coset spaces of \( G/G_{d+2} \) are also finite coset spaces of \( G \), and by \cref{thm:nilp.coset} the finite coset spaces of this quotient are not uniformly $(1/d)$-almost flat.

\end{proof}

\section{Abelian-by-$\Z$ groups}\label{sec:ab-by-z}

In this section we prove \cref{thm:poly} for groups of the form $\Z^n\rtimes\Z$.

\subsection{Some results about semidirect product} \label{subsec_semi_direct}
Let $G = K \rtimes_\P H$ with  $\Phi : H \rightarrow \Aut(K)$. With a slight abuse of notation, we often identify $(1_K,h) \in G$ as $h \in H$, and  $(k,1_H) \in G$ as $k \in K$. Similarly, given a subgroup $K'$ of $K$,  we often identify the subgroup $\{ (k,1_H) \mid k \in K'\}$ of $G$ as $K'$. Also, given $H' \subg H$, we often write $K \rtimes_\P H' $ to be the semidirect product defined by $\Phi|_{H'} : H' \rightarrow \Aut(K)$.
We often omit the map  $\Phi$  if it is not needed for the proof.

\begin{lem}  \label{subgroup_of_semidireict_product}
Let $G = B \rtimes_\Phi A$. Let $J $ be a subgroup of $ A$ with index $m$ and $K $ be a subgroup of $B$ with index $r$. Suppose 
for every $j \in J $ and every $ k \in K, \Phi (j)(k) \in K.$ Then we have
\begin{enumerate} [(i)]
\item $K \rtimes_\Psi J$ is a subgroup of $ G $, where we define $\Psi(j)(k) = \Phi(j)(k)$ for $j \in J$ and $k \in K$;
\item $|G:K \rtimes_\Psi J| = rm$.
\end{enumerate}
\end{lem}

\begin{proof}
\mbox{}
\begin{enumerate} [(i)]
\item 
Firstly, we will show that every $\Psi(j) \in \Aut(K)$. Clearly, $\Psi(j)$ is an injective homomorphism because $\Phi(j)$ is. For surjectivity, let $k' \in K$, then we have $\Phi(j\i)(k') \in K$ with $\Phi(j) \big(\Phi(j\i)(k')\big) = \Phi(j)\Phi(j\i)(k') = k'.$  To see the group $K \rtimes J$ is closed. Let $(k_1,j_1), (k_2,j_2)  \in K \rtimes J$. Then 
$(k_1,j_1) (k_2,j_2) = (k_1\cdot  \Phi(j_1) (k_2), j_1  j_2) \in K \rtimes J.$

\item 
Let $j_1, \ldots j_m$ be a set of  coset representatives of $J$ in $A$ and $k_1, \ldots k_r$ be a set of the coset representatives of $K$ in $B$. We claim that the set $\{ (k_t,j_s) | 1\leq s \leq m, 1 \leq t \leq r \}$ forms a complete set of coset representative of $K \rtimes_\Psi J $ in $ G $. Let $(b,a) \in G$. Then $a =j_s    j$, $b=k_t  k$ for some  $s \in \{1, \ldots m \}$, $t \in \{1, \ldots, r \}$ ,  $j \in J$, and $k \in K$. Then  
\[(b,a) = (k_t  k,j_s    j) = (k_t\cdot   \Psi(j_s)\Psi(j\i_s) k,j_s \cdot  j) = (k_t ,j_s  ) \cdot  (\Psi(j\i_s) k,j).\]

We also need to show that there is no redundant element, this is clear because 
$K \rtimes_\Psi J \ni (k_t,j_s)(k_v,j_u)^{-1} = (k_t,j_s)(\Psi(j\i_u)(k_v \i), j\i_u) $ implies that $j_s =j_u$ and hence $k_t=k_v$.
\end{enumerate}
\end{proof}

\begin{lem}  \label{semidirect_group_normal_subgroup}
Let $G = B \rtimes_\Phi A$. Let $J \nsubg A$ and  $K \nsubg B$, such that  
\begin{enumerate} [(i)]
\item $\forall  j \in A$ and $\forall k \in K$, $\Phi (j)(k) \in K$ (this is equivalently to $K \nsubg G$);

\item  $\forall j \in J$, and $\forall b \in B$, we have  $\Phi(j)(b) \in b K $; equivalently, for all $j \in J$, $\Phi(j): B \rightarrow B$ projects the identity map  $ \bar{\Phi}(j):\frac{B}{K} \rightarrow \frac{B}{K}$.

\end{enumerate} 

Then,
\begin{enumerate} [(i)]
\item The subgroup $K \rtimes J$ is  normal in $G$;
\item The semidirect product $\frac{B}{K} \rtimes_\Psi \frac{A}{J} $, with $\Psi(a J)(b K) = \Phi(a)(b)K$ is  well defined;
\item In particular, we have the following isomorphism:
\[ \frac{G}{ K \rtimes_\Phi J }  \cong \frac{B}{K} \rtimes_\Psi \frac{A}{J}.  \]
\end{enumerate}

\end{lem}

\begin{proof}
\mbox{}
\begin{enumerate} [(i)]
\item  Let $(k,j) \in K \rtimes J$, $(b,a)\in B \rtimes A$. Then 
\begin{align*}
(b,a)(k,j)(\Phi( a \i)( b \i), a \i)&= (b\cdot \Phi(a)(k),a  j) (\Phi( a \i)( b \i), a \i) \\
&= ( b\cdot \Phi(a)(k)\cdot \Phi(a  j)\Phi( a \i)( b \i),a  j a \i) \\
&=(b\cdot \Phi(a)(k)\cdot \Phi(j)(b\i) ,a  j a \i)
\end{align*}
From (ii), we can deduce that $\Phi(j)(b\i) =   b \i  {k'}  $ for some $k' \in K$. Therefore the first component is indeed in $K$.

\item We want to show that $\Psi$ is a well-defined action, i.e. given  $a \in A$, we want to show that $\Psi(a J) \in \Aut\left(\frac{B}{K}\right)$. We will first show that $\Psi(a J)$ is well-defined and injective. For $b_1,b_2 \in B$, we have$$
\Psi(a J) (b_1  K) =\Psi(a J) (b_2  K)
\Leftrightarrow \Phi(a)(b_1 b_2 \i) \in K
\Leftrightarrow  \Phi( a \i) \Phi(a) (b_1 b_2 \i) \in K
\Leftrightarrow b_1K = b_2  K.
$$
For sujectivity, given $bK \in B/K$, we have $\Phi(a J)(\Psi( a \i)(b) K)=bK$. To show that $\Psi$ is a homomorphism, we have 
\begin{align*}
\Psi(a J) (b_1K)(b_2K)  &= \Psi(a J) (b_1 b_2K) \\
&= \Phi(a) (b_1 b_2)K \\
&= \big(\Phi(a) b_1\cdot \Phi(a)b_2\big)K \\
&= \Phi(a) b_1 K \cdot \Phi(a)b_2K \\
&= \Psi(a J) (b_1 K) \cdot \Psi(a J) (b_2 K).
\end{align*}
Next, we will show that $\Psi: \frac{A}{J} \rightarrow \Aut \left( \frac{B}{K} \right)$ is a well-defined map. Suppose $a_1  J = a_2  J$, then $  a_2   a_1 \i\in J.$
By assumption (ii), we have  $   \Phi(a_1)(b) K=  \Phi (a_2   a_1 \i) \Phi(a_1)(b) K=  \Phi(a_2)(b)  K ,$
 and so $\Psi(a_1 J) (bK) = \Psi(a_2 J) (bK)$.

\item
We will show the following map is an isomorphism:
\begin{align*}
\Pi : \frac{G}{K\rtimes J} &  \longrightarrow \frac{B}{K} \rtimes_\Psi \frac{A}{J} \\
 (b, a) K\rtimes J &\mapsto (b K, aJ).
\end{align*}
Firstly, we will show that $\Pi$ is well-defined and injective:
\begin{align*}
  (b, a) K\rtimes J =  (b',a') K\rtimes J   &  \Leftrightarrow   (b \cdot \Phi(a{a'} \i)(b') \i,a {a'} \i) \in  K\rtimes J;\\
  & \Leftrightarrow  a {a'} \i \in J \text{ and } b \cdot \Phi(a {a'} \i)(b') \i \in K;\\
  & \Leftrightarrow   a {a'} \i \in J \text{ and } b  {b'} \i \in K;\\
  & \Leftrightarrow   (bK,a J) =  (b'K,a' J).
\end{align*}
Next, we will show that the map $\Pi$ is a homomorphism:
\begin{align*}
\Pi \bigg( (b,a) K\rtimes J \cdot   (b',a') K\rtimes J \bigg) 
&=\Pi \bigg ( (b \cdot  \Phi(a)(b'),a {a'} \i)  K\rtimes J  \bigg) \\
&=(b \cdot  \Phi(a)(b') K , a {a'} \i  J) \\
&=(b  K , a  J) (b' K , a'  J).
\end{align*}

\end{enumerate}

The fact that $\Pi$ is surjective is clear from the definition. 
\end{proof}

\subsection{Some standard results from number theory}

We will introduce some lemmas from number theory which we will use later. Throughout this section, $K$ is a field with characteristic 0, and $L$ is the algebraic closure of $K$, i.e. every polynomial in $K[x]$ splits into linear factors in $L[x]$.

\begin{lem} \label{repeated_roots_implies_common_roots}
Suppose $f$ has repeated roots $\alpha$. 
Then $\a$ is also a root of $f'$.
In particular, $f$ and  $f'$ have a common root in $L$.

\end{lem}

\begin{proof}
Since $f(x)$ has a repeated roots $\alpha$,  $f'(x) = (x - \alpha)^2 g(x)$.
Therefore, by the product rule, we have 
$f'(x)= 2(x -\alpha)g +(x -\alpha)^2 g'(x).$
\end{proof}

\begin{lem} \label{common_root}
Let $f(x)$ , $ g(x) \in K[x]$. Then $\gcd(f , g ) $ is in $ K[x]$. In particular, if $\gcd(f , g ) = 1$  in  $K[x]$ then  $\gcd(f , g ) = 1$  in  $L[x]$.
 
\end{lem}

\begin{proof}
The $\gcd$ of two polynomial can be computed using Euclidean algorithm, the initial step as follows: Without loss of generality, we assume $\deg(g(x)) \leq \deg(f(x)) $. Using long division,  we get two polynomials $q(x)$ and $r(x)$ such that $f(x)=q(x)g(x)+r(x)$. Since $K$ is a field, both $q(x)$ and $r(x)$ are in $K[x]$. Then the process repeat where we divide $g(x)$ by the $r(x)$.  By induction we can deduce that  $\gcd(f, g ) \in K[x]$. 
\end{proof}

\begin{lem} \label{irr_implies_no_common_roots}
Suppose $f(x) \in K[x]$ is an irreducible, monic and non-constant polynomial. Then $f(x)$ and $f'(x)$ have no common roots in $L$.
\end{lem}

\begin{proof}
Suppose  $f(x)$ and $f'(x)$ have a common roots in $L$. 
Then $\gcd(f , f' ) \neq 1 $ in  $L[x]$.
By  \cref{common_root}, $\gcd(f , f' ) \neq 1 $ in  $K[x]$.
Also, since $\gcd(f, f') \in K[x]$ divides $f$ and $f$ is monic irreducible, we must have $\gcd(f , f' ) =f$. Hence $f$ divides $f'$.
However, $\deg(f')\leq \deg(f)$, so $f'=0$.
Since $K$ has characteristic $0$, $f$ having zero derivative implies that it must be a constant. This contradicts the fact that $f$ is non-constant.
\end{proof}

\begin{lem} \label{split_over_inf_primes_with_no_zero_roots}
Let \( P(x) \in \mathbb{Z}[x] \) be a monic polynomial with \( P(0) \ne 0 \). Then the set of primes \( p \) for which \( P(x) \) splits completely over \( \mathbb{F}_p \) and has no zero root has positive density. In particular, there are infinitely many such primes.
\end{lem}

\begin{proof}
The fact that the set of primes \( p \) for which \( P(x) \) splits completely over \( \mathbb{F}_p \) has positive density follows from the Frobenius Density Theorem (see~\citep{dasfrob} or~\citep[Page~11]{Frobenius_eng}). Moreover, for every such prime \( p > |P(0)| \), it is clear that \( P(x) \) has no root at zero modulo \( p \), since \( p \nmid P(0) \).
\end{proof}

Given a monic polynomial $P(x)$ in $\Z[X]$, we will denote $\Pr(P)$ the set of primes $p$ such that $P(x)$ splits over $\F_p$ with no zero roots. Given $p \in \Pr(P)$, we denote $\l(P,p)$ to be the lowest common multiple of the multiplicative order of the set of  roots of $P(x)$ in $\F_p$.

Let $R$ be an integral domain, we will also make use of the function  $res : R[x] \times  R[x] \rightarrow R$,  the \emph{resultant} of two polynomials. { The resultant $res(f,g)$ of $f$ and $g$ is defined as the determinant of their Sylvester matrix, which is a matrix whose entries are either zero or the coefficients of $f$ and $g$. One consequence of the definition is that, given $f$ , $ g \in \Z[x]$ and a prime $p$, we have $ \res(f \bmod{p}, g \bmod{p} ) = \res(f,g) \bmod{p}$.} We refer the readers to \citep[\S12]{res} for the relevant details. Here is an important property of this function we will use later.

\begin{thm}[{\citep[\S12]{res}}] \label{res}
Let $R$ be  an integral domain, let $\operatorname {Frac} (R)$ be its  field of fractions.
Given $R$ and  $f(x)$ , $ g(x) \in R[x]$, the following two statements are equivalent:

\begin{enumerate} [(i)]
\item $res(f,g) = 0.$
\item $f$  and  $g$  have a common root in the algebraic closure of  $\operatorname {Frac} (R)$. \
\end{enumerate}
\end{thm}

\begin{cor}
 \label{irr_monic_implies_distinct_roots_eventually}
Let $P(x)$  be a monic irreducible polynomial in  $ \Z[X]$. Then for all but finitely many $p \in \Pr(P)$, $P(x)$ has distinct roots over $\F_p$.
\end{cor}
\begin{proof}
 Since $P$ is irreducible over $\Z$, by Gauss's Lemma, it is also irreducible over  $\Q$. \cref{irr_implies_no_common_roots} tells us that  $P(x)$ and $P'(x)$ have no common roots in $\C$. By  \cref{res}, $res(P,P')  \in \Z \setminus \{0\}$.   In particular, for primes $p > res(P,P') $, we have $res(P,P') \not\equiv_p 0  $. Hence, $P(x)$ and $P'(x)$, when viewed as polynomials over $\F_p$, have no common roots (even in their splitting field). \ Therefore, \cref{repeated_roots_implies_common_roots} implies all roots of $P(x)$ are distinct.
\end{proof}
\begin{lem} \label{product_of_poly}
Let $\F$ be a field.
Let $f(x), g(x)$ be two monic non-zero polynomials in $ \F[x]$ such that  $f(x)g(x)$ splits over $\F[x]$. Then both $f(x)$ and $g(x)$ split over $\F$. \
\end{lem}

\begin{proof}
Since $\F$ is a field, $\F[x]$ is a unique factorisation domain, from which the result follows immediately.
\end{proof}

\subsection{The multiplicative order of the eigenvalues}  \label{subsec:mutliplicity_of_eigenvalue}
 Let $M \in \GL(n,\Z)$, we will denote $\ch(M)$ the characteristic polynomial of $M$, which is a degree $n$ monic polynomial in $\Z[x]$.  We denote $\Pr(M)$ to be $\Pr(\ch(M))$ and $\l(M,p)$ to be $\l(\ch(M),p)$.  Finally, we denote $\ord(M)= \inf \{n \in \N : M^n = I_n\}$, and similarity $\ord_p(M) = \inf\{n  \in \N : M^n \equiv_p I_n  \}$. Let $(p_k)_{k=1}^{\infty}$  be the sequence of strictly increasing primes such that $\ch(M)$ splits and has no zero roots.  Our aim is to prove that  $\l(M,p_k) \rightarrow \infty $  as $k \rightarrow \infty$. To simplify the notation, we will often drop the subscript $k$ and say  $\l(M,p) \rightarrow \infty $  as $p \rightarrow \infty$.

\begin{lem} \label{diag_case}
Let  $M \in \GL(n,\Z)$  such that $\ord(M) = \infty$.
Let $q$ be a prime such that $M$ is diagonalizable over $\F_q$.
Then $\l(M,q) \geq \log_{\n{M}_2} (q-1)$.
\end{lem}

\begin{proof}
Since $M$ is diagonalizable over $\F_q$, the eigenvectors form a basis over $\F_q^n$. We can therefore deduce that $\ord_q(M) =  \l (M,q)$.
 Since $\ord(M) = \infty$, for all $k \in \N$, $M^k  \neq I_n$. Hence, there exist $i \in \{1, \ldots, n \}$, such that $M^{\ord_q(M)} e_i \in \{e_i +(q\Z)^m | m \in \Z \} \backslash \{e_i\}$, so $M$ send $e_i$ with norm $1$ to $M^{\ord_q(M)} e_i$ with norm at least norm at least $q-1$ in $\ord_q(M)$ steps. This in particular implies that $\n{M}_2 > 1$. Furthermore, $\ord_q(M)  \geq \log_{\n{M}_2} (q-1)$.
\end{proof}

Given a commutative ring $R$, every monic polynomial $P(X) \in R[X]$  is a characteristic polynomial of some matrix with entries in $R$; indeed, one such matrix is thc \emph{companion matrix}, constructed in \citep[\S3.3]{matrix}. 
\begin{lem} \label{irr_implies_diverges}
Let $P(x)$  be a monic irreducible polynomial in  $ \Z[x]$ that has a root with absolute value not equal to $1$. 
Then $\l(P,p)\rightarrow \infty$  as $p \rightarrow \infty$ {through $\Pr(P)$}.
\end{lem}
\begin{proof}
Since $P(x)$ is  irreducible over  $\Z$,  \cref{irr_monic_implies_distinct_roots_eventually} tells us that for big enough $p \in \Pr(P)$, all roots of $P(x)$ are distinct. For every such prime $p$, the companion matrix of $P$ is diagonalisable over $\Fp$, so the result follows from \cref{diag_case}.
\end{proof}
\begin{cor} \label{diverges_mul_order}
Let $P(x)$  be a monic polynomial in  $ \Z[x]$ that has a root with absolute value not equal to $1$. 
Then $\l(P,p) \rightarrow \infty$  as $p \rightarrow \infty$ {through $\Pr(P)$}.
\end{cor} 
\begin{proof}
Since $\Z[x]$ is a unique factorization domain, we can write $P(x)= \prod_{i=1}^r f_i(x)$, a product of monic irreducible polynomials. Without loss of generality, $f_1(x)$ has a root with absolute value not equal to $1$. According to \cref{product_of_poly}, $\Pr(P) \subeq \Pr(f_1)$. The result follows from \cref{irr_implies_diverges}.
\end{proof}

\subsection{$\Z^n$  by $\Z$ groups } \label{subsection_uniform_d_inzn_by_z}\label{section_characterise_Zn_by_Z}
We will first recall some linear algebra results.  Suppose that $T$ is a  linear map over  $V = \F^n$.   We define the \emph{ eigenspace} of $T$ for $\l$ as
 $ E_{\l}(T) = \{ x \in   \F^n |(  T -  \l   I)  x = 0 \}. $
 We define the \emph{generalized eigenspace} of $T$ for $\l$ as
 $ \GE_{\l}(T) = \{ x \in   \F^n |(  T -  \l   I)  ^n x = 0\}$.  Let $A \in \GL(n, \F)$ such that $\ch (A)$ splits over $\F$. Let $\l$ be an eigenvalue of $A$. 
Then $v \in \GE_\l (A)$ is called a \emph{last generalised eigenvector} of $A$ with respect to $\l$ if there is no $u \in  \GE_\l(A)$, such that $(A- \l I)u = v$. If the characteristic polynomial of the linear map $T$ splits over $\F$, then we know $T$ has an eigenvector, and hence a last generalised eigenvector. In particular, we have the following observation.

\begin{lem} \label{eigenvector_in_invariant_subspace}
Suppose $char(A)$ splits over $\F$. Let $W$ be a $A$-invariant subspace. Then $W$ contains an eigenvector of $A$.
\end{lem}

Let $p$ be a prime. 
We will   use the notation $x \mapsto \tilde{x}_p$ for the quotient map from $\Z$ to $\Z_p$.
Given  a vector $v$ in $\Z^n$, we will denote $\tilde{v}_p \in \Z_p^n$ for $v \bmod p$; similarly given a matrix  $M \in \GL(n,\Z)$, we denote $\ti{M}_p \in \GL(n,\Z_p)$ for $ M\bmod p$. When it is clear from the context which prime we are referring to, we often drop the subscript $p$ for the matrix $\ti{M}_p$ and the vector $\tilde{v}_p$.

\begin{lem} \label{index_p_subgroup}
Let $M \in \GL(n,\Z)$ and $p$ be a prime such that $\ch(M)$ splits over $\F_p$.
Let $v \in \Z^n$ such that $\tilde{v}$ is a last eigenvector of $\tilde{M}$ with respect to an eigenvalue $\tilde{\l}$.
Let $B \subset \mathbb{Z}^n$ be such that $\tilde{B}$ is a basis of $\mathbb{Z}_p^n$ consisting of generalized eigenvectors of $\tilde{M}$ and containing $\tilde{v}$.
Define $A_v = \langle B \setminus \{v\}, p e_i | 1 \leq i \leq n \rangle$, a subgroup of $\Z^n$. Then,

\begin{enumerate} [(i)]
\item The subgroup $A_v$ has index $p$ in $\Z^n$, and is invariant under $M$, i.e. for all $h\in A_v < \Z^n$, we have $M(h) \in A_v$.
\item Let $\bar{M}$ be the projected endomorphism of $M$ on $\frac{\Z^n}{A_v} \iso \Z_p$, then for every $ w \in \frac{\Z^n}{A_v}$, when viewed as an element in $\Z_p$, we have  $\bar{M}(w ) \equiv_p \l $  . 
\end{enumerate}

\end{lem}

\begin{proof}
\mbox{}

\begin{enumerate} [(i)]
\item
By an isomorphism theorem, the following  quotient groups are isomorphic:
\begin{equation} \label{iso_quotients}
\frac{\Z^n}{A_v} \cong \frac{\Z^n / (p\Z)^n}{ A_v /(p\Z)^n} \iso \frac{(\Z_p)^n}{ \langle \tilde{B} \setminus \{\tilde{v}\}   \rangle} \iso \langle \tilde{v} \rangle \iso \Z_p.
\end{equation}
Hence, the subgroup $A_v \leqslant \Z^n$ has index $p$. We will next show that  $A_v $ is invariant under $M$: since $\tilde{v}$ is a last eigenvector of $\tilde{M}$, we have $\tilde{M} (\langle \tilde{B} \setminus \{\tilde{v}\}   \rangle ) \subeq \langle \tilde{B} \setminus \{\tilde{v}\}   \rangle.$ Therefore, $M (\langle B \setminus \{v\}\rangle) \subeq \langle B \setminus \{v\}\rangle + (p\Z)^n.$  Also, since  $ (p\Z)^n$ is characteristic in $ \Z^n $ and $M$ is a homomorphism on $\Z^n$, we have $A_v = \langle B \setminus \{v\}, p e_i | 1 \leq i \leq n \rangle$ is invariant under $M$.

\item This follows from \eqref{iso_quotients} and the fact that  $\tilde{v}$ is a last eigenvector of $\tilde{M}$ with respect to  $\tilde{\l}$.\qedhere

\end{enumerate}
\end{proof}

\begin{lem} \label{normal_subgroup_contain_Zp}

Let $p$ be a prime. Let  $\G = \Z_p \rtimes_\p \Z_r$ be the semidirect product where $\p (1)(1) = m \in \Z_p \setminus \{0\} $ and  $r = \ord_p(m)$.  Let $N$ be a normal subgroup of $\G$. Suppose $\frac{\G}{N}$ has a nilpotent subgroup with index less than $ r$. Then $N \supseteq \Z_p \times \{ 0 \}$.
\end{lem} 

\begin{proof}
Let  $\frac{\G'}{N}$ be the nilpotent subgroup of $\frac{\G}{N}$ with index $< r$.  Then, we have \[ \frac{rp}{ |\G'|} = \frac{|\G|}{ |\G'|}  = \frac{|\G /N| }{ |\G'/N|} <r. \]
Hence we have $p<|\G'|$. 
Also, since $|\G'| \mid rp $ and $r \leq  p-1$, we have $p \mid |\G'|.$ 
Hence there exists an element $(t,s) \in \G'$ with order $p$. It follows that $(0,0)= (t,s)^p = (x,ps) = (x,s)$ for some $x \in  \Z_p$.  Hence $s= 0$, so we have $\Z_p  = \langle (t,0) \rangle \subseteq \G'$. Also, $\Z_p \subsetneq \G'$, so there exists $u \in \Z_r \setminus \{0\}$,  $v \in \Z_p$, such that $(v,u) \in \G'$. Hence, $(0,u)=(-v,0)(v,u) \in \G'$. Hence,  $\Z_p \rtimes \la u\ra \subeq \G'$.   We claim that the subgroup  $[ \{ 0 \} \times \la u\ra  , \Z_p \times \{ 0 \} ] $  contains $ \Z_p \times \{ 0 \} $.
We have $\p(u)(1) = m ^{u} \not \equiv_p 1 .$  It follows that $[ \{ 0 \} \times \Z_r  , \Z_p \times \{ 0 \} ] \ni (0,u ) (1,0) (0,-u) (-1,0) = (m ^{u} -1, 0)$. Hence,  $[ \{ 0 \} \times \Z_r  , \Z_p \times \{ 0 \} ] \supseteq \langle  (m ^{u} -1, 0) \rangle =  \Z_p \times \{ 0\}$, this proves the claim. Therefore, by induction, we can deduce that every term in the  lower central                                                                                                                                                                                                                                                                                                                                                                                                                                                                                                                                                                                                                                                                                                                                                                                                                                                                                                                                                                                                                                                                                                                                                                                                                                                                                                                                                                                                                                                                                                                                                                                                                                                                                                                                                                                                                                                                                                                                                                                                                                                                                                                                                                                                                                                                                                                                                                                                                                                                                                                                                                                                                                                                                                                                                                                                                                                                                                                                                                                                                                                                           series of $\G$ contains $[ \{ 0 \} \times \Z_r  , \Z_p \times \{ 0 \} ] \supseteq   \Z_p \times \{ 0\}$. Since the quotient $\frac{\G'}{N}$ is nilpotent,  $ N$ must contains a term in the lower central of $\G$. Therefore, we have $\Z_p \subseteq N$.
\end{proof}

The following theorem is an alternative form of \cref{thm:bt}.
\begin{thm}[{\citep[Theorem 4.1]{nil}}]\label{matt_lemma} Let $\varepsilon, \delta>0$. Then there exist constants $C_{\epsilon,\delta}$ and  $D_{\epsilon,\delta}$ depending only on $\epsilon,\delta$ such that the following holds. Suppose $G$ is a finite group generated by a symmetric subset $S$ containing the identity whose Cayley graph has diameter $\gamma:=\diam_S(G)$ and satisfies
 \[\gamma \geq \left( \frac{|G|}{|S|} \right)^\epsilon\]
and $\gamma \geq D_{\epsilon,\delta}$. Then
 $G$ has a normal subgroup $H$ contained in $S^{\lfloor\gamma^\delta\rfloor}$ such that $G/H$ has a nilpotent subgroup of index  at most $C_{\epsilon,\delta}$.
\end{thm}

We will be using the same set-up as \citep[\S14]{ggt}. The semidirect product $\Z^n  \rtimes_\Phi\Z $ is defined by a homomorphism $\Phi : \Z \rightarrow Aut(\Z^n) = \GL(n, \Z)$, and $\phi$ is entirely determined by the automorphism $\t \coloneqq \phi(1)$, which can be represented by a matrix $M \in \GL(n,\Z)$. We can therefore denote the  semidirect product  $\Z^n  \rtimes_\phi\Z $ as $\Z^n \rtimes_M \Z$.

\begin{prop}[uniformly   $\alpha$-almost flat coset spaces implies virtual nilpotence in  $ \Z^n$  by $\Z$ ] \label{characterise_Zn_by_Z}
Let $G = \Z^n   \rtimes_M \Z $. Then $G$ has uniformly   $\alpha$-almost flat coset spaces implies all eigenvalues of $M$ have absolute value $1$.

\end{prop}

The result follows immediately from the following more precise proposition.

\begin{prop}   \label{set_of_subgroups}
Let $G = \Z^n   \rtimes_M \Z $ where $M$ has an eigenvalue with absolute value not equal to $1$.  Let $\cp$ be an infinite set of primes $p$ such that $\ch(M)$ splits over $\F_p$  with no zero roots. For each $p \in \cp$, let $\tilde{v}_p$ be a last eigenvector of $\tilde{M}$ with respect to an eigenvalue $\tilde{\l}$ with the largest multiplicative order over $\F_p$. Let $A_p = A_{v_p}$ be an $M$ -invariant subgroup of $\Z^n$ of index $p$ that satisfies the properties in  \cref{index_p_subgroup}. Consider the set of subgroups 
$\H_\cp = \{ A_{p} \rtimes r \Z \mid  p \in \cp, r= 1 \text{ or }  \ord_p(\lambda) \}.$ 
Then for all $\a \in (0,1]$ and $c>0$ there exists $H \in \H_\cp$, such that $\diam(G/H) < c|G:H|^\alpha$.
\end{prop}

\begin{proof} 
We will prove the statement by contradiction. Let $S$ be a generating set of $G$. Let $\a \in (0,1]$ and $c>0$, and suppose that $M$ has an eigenvalue with absolute value not equal to $1$ but that 
\begin{equation} \label{contradiction_uniform}
    \diam(G/H) \ge c|G:H|^\alpha
\end{equation}
for all $H \in \H_\cp$. 
Fix $\epsilon < \frac{\alpha}{2}$ and $\delta = \frac{\epsilon}{2}$. Given a prime $p$ such that $\ch(M)$ splits over $\F_p$ with no zero roots. Let $\lambda$ be an eigenvalue of $M$ over $\F_p$ with the largest multiplicative order, and let $r_p = \ord_p(\lambda)$.   
Let $A_p$ be an $M$ -invariant subgroup of $\Z^n$ of index $p$ that we found in  \cref{index_p_subgroup} and let  $H_p = A_p \rtimes r_p\Z$. By  \cref{semidirect_group_normal_subgroup}, $H_p$ is a normal subgroup of $G$, and $\G_p \coloneqq \frac{G}{H_p} \iso \Z_p \rtimes \Z_{r_p}$. Let $\bar{S}$ be the image of $S$ in $\G_p$.  By hypothesis we have that $ \diam(\Gamma_p) \geq  c|\G_p|^\a = c|r_p p|^\a$, and hence, for all large enough $p$, that $ \diam(\Gamma_p) \geq |r_p p|^{\a/2}$. \
 We will now apply \cref{matt_lemma}. By taking $p$ big enough, we have  $\diam(\Gamma_p) \geq D_{\epsilon,\delta}$, and so there exists $N_p \nsubg \G_p$, such that  
\begin{enumerate} [(i)]
\item $N_p \subseteq \bar{S}^{[\diam(\Gamma_p)^\delta]}$;
\item $\frac{\G_p}{N_p}$ has a nilpotent subgroup with index bounded  by $C_{\epsilon,\delta}$. 
\end{enumerate}
By  \cref{diverges_mul_order}, for big enough $p$, we have $r_p> C_{\epsilon,\delta}$. Therefore, by  \cref{normal_subgroup_contain_Zp}, we have $N_p \supseteq \Z_p \times \{ 0 \}$.
Moreover, trivially
$ \diam(\Gamma_p) \leq |G:H_p| = r_pp \leq p^2$.
Hence, $\diam(\Gamma_p) ^\delta \leq p^{2\times \frac{\epsilon}{2}} \leq p ^\epsilon$.
By the choice of $N_p$, we have $\Z_p \times \{ 0 \}  \subseteq N_p  \subseteq \bar{S}^{[\diam(\Gamma)^\delta]} \subseteq \bar{S}^{p^\epsilon}  $.
Finally, take $H' = A_p \rtimes\Z$, we can see that $S^{p^\epsilon} H'  = G$.
Therefore, $\diam \left( \frac{G}{H'}  \right) \leq p^\epsilon = |G:H'|^\epsilon<|G:H'|^{\a/2}$, contradicting \eqref{contradiction_uniform}.
\end{proof}

\begin{cor} \label{classify_Zn_by_Z}
The following statements about $G \coloneqq  \Z^n   \rtimes_M \Z$ are equivalent:
\begin{enumerate} [(i)]
\item the matrix $M$ only has eigenvalues of absolute value $1$;
\item the group $G$ has polynomial growth;
\item the group $G$ has uniformly   $\alpha$-almost flat coset spaces;
\item all eigenvalues of $M$ are roots of unity.
\end{enumerate} 
\end{cor}

\begin{proof}
The equivalence (i) $\Leftrightarrow$ (ii) is proved in \citep[Proposition 14.1]{ggt}. The implication (ii) $\Rightarrow$ (iii) is proved in   \cref{poly_growth_implies_uniform_D_a}. The implication (iii) $\Rightarrow$ (i) is proved in  \cref{characterise_Zn_by_Z}. The implication (i) $\Rightarrow$ (iv) is proven in \citep[Lemma 13.28]{ggt}.
\end{proof}

\section{Nilpotent-by-$\Z$ groups} \label{section_nil_by_z}

In this section we prove \cref{thm:poly} in the special case of a nilpotent-by-cyclic group, or more precisely a group of the form $G \rtimes_\P \Z$ with $G$ nilpotent. This group is determined by the homomorphism $\P : \Z \rightarrow \Aut(G)$, and the map $\P$ is entirely defined by the automorphism $\P(1) \eqqcolon \p \in \Aut(G)$, so we will often write $G \rtimes_\P \Z \eqqcolon G \rtimes_\p \Z$.

Recall that the set of all finite order elements forms a characteristic subgroup of $G$, we will denote it by $\Tor G$. Suppose $G$ is $k$-step nilpotent, i.e.  $G$ has the lower central series 
$G= G_1 \geqslant G_2 \geqslant \ldots G_{k+1} = \{\id_G\},$ where $G_{i+1} = [G_i,G]$. 
Let $\phi_i \coloneqq \phi \vert_{G_i}$ be the restriction of $\phi$ on $G_i$. Then $\phi_i$ projects to an automorphism $\t_i$ on the finitely generated abelian group $B_i \coloneqq \frac{G_i}{G_{i+1}}$. The automorphism $\t_i$ then projects to an  automorphism $\bar{\t}_i$ on the torsion free quotient $\frac{B_i}{\Tor{B_i}} \iso \Z^{m_i}$. 
Fix a basis for each $\frac{B_i}{\Tor{B_i}} \iso \Z^{m_i}$, then each   automorphism $\bar{\t}_i$ is corresponds to a matrix $M_i \in \GL(m_i,\Z)$.

In the proof of  characterising nilpotent by $\Z$ groups in terms of their growth rate \citep[Proposition 14.28]{ggt}, there are two cases to consider, whether or not all of $M_i$ have eigenvalues of absolute value $1$. It turns out that it is enough to examine the first matrix $M_1$.

\begin{lem} \label{first_matrix_has_e-val_1_implies_others_too}
Suppose the matrix $M_1$ only has eigenvalues of absolute value $1$. Then all matrices $M_i$ also only have eigenvalues of absolute value $1$.
\end{lem}

The idea of the proof is to look at the growth of distance under the matrices. For simplicity, we will be measuring the distance using  the infinity norm, i.e. for $v = (v_1, \ldots, v_m) \in \C^m$, we let  $\norm{v} = \max_{1\leq i \leq m}\{|v_i|\} $. For a matrix $M \in \GL(m,\C)$, we denote $\norm{M}$ the operator norm of $M$. Given a function $f:\N \longrightarrow \ \R$, we say $f(n)$   has \emph{polynomial growth} if there exists $C,d > 0$, such that $f(n) \leq Cn^d$ for all $n \in \N$;  likewise, we say $f(n)$ has \emph{exponential growth} if there exists $\a> 1$, such that $f(n) \geq \a^n$ for all $n \in \N$. We denote $e_1, \ldots, e_m$ the standard basis of $\C^m$.

\begin{lem} \label{eigenvalue_and_growth}
Let $M \in \GL(m,\C)$. Then one of the following holds.
\begin{enumerate} [(i)]
\item For all $v \in \C^m$ we have that $\norm{M^n v}$ has polynomial growth (when all eigenvalues of $M$ have   absolute value $1$). 
\item There exists $u \in \Z^m$ such that $\norm{M^n u}$ has exponential growth (when $M$ has at least one eigenvalue of absolute value $ \neq 1$).
\end{enumerate}

\end{lem}

\begin{proof}
\mbox{}
\begin{enumerate} [(i)]
\item

Let $J= P^{-1}MP$ be the Jordan normal form of $M$ for some change of basis matrix $P$. By looking at the $n-$th power of $J$, we can deduce that there exists $B > 0$, such that every entry of $J^n$ is bounded by $B n^{m}$. In particular,  $\norm{J^n e_i}$,the maximum entry of $i-th$ column, is  bounded by $ B n^{m}$. Let $v \in \C^m$, then there exists $\a_1, \ldots, \a_m$, such that $P^{-1}v = \sm \a_i e_i$. Then we have
$$
\norm{M^n v} =    \norm{P J^n  \sm \a_i e_i} \leq \norm{P} \sm |\a_i| \norm{ J^n e_i} \leq \left(\norm{P}  \sm |\a_i|  B\right) n^{m}.
$$

\item
Let $v = \sm \a_i e_i \in \C^m$ be a eigenvector  of $M$ corresponds to the eigenvalue $\lambda$ with $|\l| >1$, then we have  
\[|\l|^n \norm{v}
= \norm{M^n v }
 = \norm{\sm \a_i M^n e_i }
 \leq \sm |\a_i| \norm{ M^n e_i }
\]
Since $|\l|^n \norm{v}$ has exponential growth, there must exist $i \in \{1, \ldots, m \}$ such that  $\norm{ M^n e_i }$ has exponential growth. \qedhere
\end{enumerate}
\end{proof}

\begin{lem} \label{poly_growth_implies_poly_growth}
Let $B : \C^p \times \C^q \longrightarrow \C^r$ be a bilinear map.
Let $u(n) = (u_1(n), \ldots, u_p(n))$ be a sequence of vectors in $\C^p$, and  $v(n) = (v_1(n), \ldots, v_q(n))$ be a sequence of vectors in $\C^p$ such that both have $\n{u(n)}$ and $\n{v(n)}$ polynomial growth, then $\n{B(u(n),v(n))}$ also has polynomial growth.
\end{lem}

\begin{proof}
By bilinearity of $B$, we have 
\[ \norm{B ( u(n),v(n))}= \norm {\sum_{i,j}u_i(n)v_j(n) B(e_i, e_j) }  \leq \sum_{i,j} |u_i(n)v_j(n)| \norm {B(e_i, e_j) }  \]
which, as a sum of products of polynomials in $n$, is still a polynomial  function in $n$.
Hence, $\n{B(u(n),v(n))}$ also has polynomial growth.
\end{proof}

\begin{lem} \label{torsion_quotient_well_defined}
Let $\G_1$, $\G_2$ and $\G_3$ be finitely generated  nilpotent groups. Suppose $ \psi : \G_1 \times \G_2 \longrightarrow \G_3 $ is a homomorphism in each variable. Then the projected map  $ \bar{\psi} : \frac{\G_1}{\Tor{\G_1}} \times \frac{\G_2}{\Tor{\G_2}} \longrightarrow \frac{\G_3}{\Tor{\G_3}} $ is also a homomorphism in each variable.
\end{lem}

\begin{proof}
We will first show that $\bar{\psi}$ is well defined.
 Let $g$ ,  $g' \in \G_1$ and $h$, $h' \in \G_2$ such that $g \Tor{\G_1} = g'\Tor{\G_1}$  and $h \G_2 = h'\G_2$, i.e. we have $g\i g' \in \Tor{\G_1}$ and $h \i h' \in \Tor{\G_2}$.   
 Therefore $\psi(g,h)  \Tor{\G_3} =  \psi(g,h) \psi(g,h \i h') \psi(g\i g' ,h') \Tor{\G_3} =  \psi(g',h')  \Tor{\G_3}$  which shows $\bar{\psi}$ is well defined.
The fact that $\bar{\psi} $ is a homomorphism in each variable follows from the assumption the $\psi$  is a homomorphism in each variable.
 \end{proof}

Next, we will  view  every element in  $\frac{B_i}{\Tor{B_i}}$ as an element in $ \Z^{m_i} \subseteq \C^{m_i}$, and identity  $\bar{\t}_i$ as the matrix $M_i$. In particular, the map  $\bar{\t}_i$ on $\frac{B_i}{\Tor{B_i}} \iso \Z^{m_i}  $ can be uniquely extended to a linear map on $\C^{m_i}$, with slight abuse of notation, we will also denote this extended map by $\bar{\t}_i$.

\begin{proof}[Proof of  \cref{first_matrix_has_e-val_1_implies_others_too}]

We will prove by induction on $i$, the index of the matrix. The base case, $M_1$ only has eigenvalues of absolute value $1$, is given as an assumption of the lemma. Suppose all eigenvalues of $M_{i-1}$ have absolute value $1$. Then by  \cref{eigenvalue_and_growth}, we know that for any  $v \in \Z^m_i$, $\norm{M_{i-1}^n v}$ has polynomial growth. Let $g  \in G$ and $h \in \Gm$, we will show that sequence of the form $\n{\bar{\t}_i^n([g,h] \Gp \T{B_i} )}$ have polynomial growth. Consider the map   
\begin{align*}
\bar{\psi}_i :  \frac{B_1}{\Tor{B_1}} \times \frac{B_{i-1}}{\Tor{B_{i-1}}}  &  \longrightarrow \frac{B_i}{\Tor{B_i}}\\
 (gG_2 \Tor{B_1}, h G_i \Tor{B_{i-1}})&\mapsto [g,h]\Gp \T{B_i}. 
\end{align*}
By  \cref{lem:multilin} and  \cref{torsion_quotient_well_defined}, $ \bar{\psi}_i $ is  a homomorphism in each variable. Also, we have 
\begin{align*}
\bar{\t}_i^n([g,h] \Gp \T{B_i} )&= \phi^n([g,h]) \Gp \T{B_i}  &\text{by definition of $\bar{\psi}_i$,}\\
&= [\phi^n(g), \phi^n(h)]\Gp \T{B_i} &  \text{since $\phi$ is an automorphism,}\\
 &=  \sb \bigg( (\phi^n(g)G_2) \Tor{B_1} ,(\phi^n(h) G_i) \Tor{B_{i-1}}\bigg) \\
&= \sb\bigg( \bar{\t}_1^n(gG_2 \Tor{B_1}) , \bar{\t}_{i-1}^n (hG_i \Tor{B_{i-1}}) \bigg). 
 \end{align*} 
Combine this equality with  \cref{poly_growth_implies_poly_growth}, we see that $\n{\bar{\t}_i^n([g,h] \Gp \T{B_i} )}$ also have polynomial growth. 
Now, given  any elements $x \in G_i$, we have 
$x=\prod_{j=1} ^l [g_j, h_j]   \text{ for some } g_j \in G, h_j \in \Gm, \text{ and } l\in \N.$ Hence, we have
$$ \n{\bar{\t}_i^n(x \Gp \T{B_i} )} =\n{\bar{\t}_i^n \left(\prod_{j=1} ^l [g_j, h_j] \Gp \T{B_i} \right)} 
 \leq  \prod_{j=1} ^l \n{ \bar{\t}_i^n( [g_j, h_j] \Gp \T{B_i} )} 
$$
which, as a sum of functions with polynomial growth, also has polynomial growth. Hence , for any $x \in G_i$, $\n{\bar{\t}_i^n(x \Gp \T{B_i} )}$ has polynomial growth. Therefore, by  \cref{eigenvalue_and_growth}, all eigenvalues of $M_{i}$ have absolute value $1$, this concludes the proof.
\end{proof}

\begin{prop} \label{uniform_d_alpha_implies_M1_has_eigenvalue_abs_1}
Let $G$ be a finitely generated nilpotent group, let $\p \in \Aut(G)$. Define $\G \coloneqq G \rtimes_\p \Z$.  Then $\G$ has uniformly   $\alpha$-almost flat coset spaces implies all eigenvalues of $M_1$ have absolute value $1$.
\end{prop}

\begin{proof}
We will again give a contrapositive proof. Suppose $M_1 \in GL(m_1, \Z)$  has an eigenvalues with absolute value greater than $1$. In particular, this implies that $m_1 \geq 2$. Firstly note that 
\begin{equation} \label{quotient_OF_P}
 \frac{G \rtimes_\p \Z}{ G_2 \times \{0\}} \iso \frac{G}{G_{2}} \rtimes_{\t_1}  \Z   =  B_1 \rtimes_{\t_1}  \Z .
\end{equation} 
Similarly, 
\begin{equation} \label{quotient_of_B_by_Z}
\frac{B_1 \rtimes_{\t_1}  \Z}{ \Tor B_1 \times \{0\} } \iso \Z^{m_1} \rtimes_{M_1} \Z.
\end{equation}

According to  \cref{characterise_Zn_by_Z}, the semidirect product  $\Z^{m_1} \rtimes \Z_{M_1}$ does not have uniformly   $\alpha$-almost flat coset spaces. Hence, \eqref{quotient_of_B_by_Z} tells us that  $B_1 \rtimes_{\phi_1}  \Z$ does not have uniformly   $\alpha$-almost flat coset spaces; likewise,  \eqref{quotient_OF_P} shows us that the group $\G =G \rtimes_\p \Z$ does not have uniformly   $\alpha$-almost flat coset spaces.
\end{proof}

We will now combine our result with Wolf's Theorem for Nilpotent-by-$\Z$ groups.

\begin{cor} \label{characterise_nilp_by_z}
Using the   notation we introduced in the beginning of \cref{section_nil_by_z},  the following statement about $\G \coloneqq G \rtimes_\p \Z$ are equivalent:
\begin{enumerate} [(i)]
\item the matrix $M_1$ only has eigenvalues of absolute value $1$;
\item every matrix $M_i$  only has eigenvalues of absolute value $1$;
\item the group $\G$ has polynomial growth;
\item the group $\G$ has uniformly   $\alpha$-almost flat coset spaces;
\item the group $\frac{B_1}{\Tor{B_1}} \rt_{\bar{\t}_1} \Z$ has uniformly   $\alpha$-almost flat coset spaces;
\end{enumerate} 
\end{cor}

\begin{proof}
The implication(i) $\Rightarrow$ (ii)  is proved in  \cref{first_matrix_has_e-val_1_implies_others_too}. The implication (ii) $\Rightarrow$ (i) is clear. The equivalence (ii) $\Leftrightarrow$ (iii)  is proved in \citep[Proposition 14.28]{ggt}. The implication (iii) $\Rightarrow$ (iv)  is proved in   \cref{poly_growth_implies_uniform_D_a}. The implication  (iv) $\Rightarrow$ (i) is proved in  \cref{uniform_d_alpha_implies_M1_has_eigenvalue_abs_1}. Finally, the implication (i) $\Leftrightarrow$ (v)  is proved in \cref{characterise_Zn_by_Z}.
\end{proof}

\section{Polycyclic groups}

In this section we complete the proof of \cref{thm:poly}. It is well known that a torsion-free polycyclic group is poly-$\Z$ and every polycyclic group contains a normal subgroup of finite index which is poly-$\Z$.  Here is a slightly strengthened version of this statement.

 \begin{lem} \label{normal_subroup_finite_index_contain_normal}
Let $P$ be a polycyclic group and $N \nsubg P$. Then there exists $P' \nsubg\f P$ containing $N$ such that $P'/N$ is torsion-free.
\end{lem}

\begin{proof}
Note that $P/N$ is polycyclic. As mentioned above,  every polycyclic group contains a normal subgroup of finite index which is poly-$\Z$, by correspondence theorem, there exists $P' \subg \f P$, such that $P'/N \nsubg\f P/N$, and $P'/N$ is a poly-$\Z$ group.
\end{proof}

In the later subsections, we will make use the maximal nilpotent normal subgroup of a polycyclic group. Recall that the \emph{Fitting subgroup} of a group $G$ is the subgroup generated by all the nilpotent normal subgroups, i.e. $\Fit (G) = \langle N \mid N \nsubg G, \text{$N$ nilpotent} \rangle$.  For a polycyclic group $P$, the Fitting subgroup $\Fit (P)$ is nilpotent \citep[Corollary~13 on p.~15]{polycyclic}.

\begin{lem} \label{normal_subgroup_contains_fit}
Let $P$ be a polycyclic group and suppose $\Fit(P)\le N\normal P$. Then $\Fit(N) = \Fit(P)$.
\end{lem}

 \begin{proof}
Since the Fitting subgroup of a polycyclic group is nilpotent, both $\Fit(N)$ and $\Fit(P)$ are nilpotent . In particular, $\Fit(P)$ is a normal nilpotent subgroup of $N$, hence $\Fit(P)\subseteq\Fit(N)$. On the other hand, since $\Fit(N) \csubg N \nsubg P$, we have that $\Fit(N)$ is a normal nilpotent subgroup of $P$, so $\Fit(N) \subseteq \FP$.
 \end{proof}

\subsection{Some group theory results}
Let $G$  be a group and $K \nsubg G$ such that $G/K \iso \Z$. Then $G \iso K  \rtimes \Z$. Following the notations defined in the beginning of \cref{subsec_semi_direct}, we often omit the brackets when defining a successive semidirect product, i.e. let $H_1, H_2, \ldots, H_n$ be groups, we will write 
$(\ldots((H_1 \rtimes H_2) \rtimes H_3) \ldots )\rtimes H_n$ as $H_1 \rtimes H_2 \rtimes \ldots \rtimes H_n$. When it is clear from the context, we will also often identify the element $(1_{H_1}, \ldots, h_i ,\ldots , 1_{H_n}) \in  H_1 \rtimes H_2 \rtimes \ldots \rtimes H_n$  as the element $h_i \in H_i$.

We can see that every poly-$\Z$  is a  successive extension of $\Z$. Indeed, we have the following observation. 
\begin{lem} \label{write_Polyz_group_with_torsion_free_quotient_as_semi_direct_product}
Let $P$ be a group and $N \nsubg P$. Suppose $P/N$ is a poly-$\Z$ group. Then
\begin{equation} \label{poly_form}
P \iso N \rtimes \Z \rtimes \Z\rtimes \ldots \rtimes\Z.
\end{equation}
\end{lem}

\begin{proof}
Since $P/N$ is a poly-$\Z$ group, it has a subnormal series 
	\[P/N = P_0/N \containsn P_1/N \containsn \ldots P_n/N \containsn P_{n+1}/N = \{1\} \] such that for each $i$, $P_{i+1} / P_i \iso (P_{i+1}/N) / (P_1/N) \iso \Z$.
	It follows by induction that $P$ has the form in \eqref{poly_form}.
\end{proof}

In the case when we need to identify the different copies of $\Z$, we will often write $\Z_{(i)}$ for the $i-$th copy of $\Z$. For example, in \cref{write_Polyz_group_with_torsion_free_quotient_as_semi_direct_product}, let $m=h(P/N)$, we will often write the group in \eqref{poly_form} as
$P \iso N \rtimes \Z_{(1)} \rtimes \Z _{(2)}\rtimes \ldots \rtimes\Z_{(m)}$.

The following lemma is often used to construct a finite index subgroup of a polycyclic group.  

\begin{lem}  \label{big_enough_l_K_lZ_is_a_subgroup}
Let $G = K  \rtimes_\p \Z$ be a finitely generated group. Suppose $K' \subg\f K$. Then there exist $\ell\in\N$ such that $K' \rtimes_\p \ell\Z  \subg \f G$. In particular, for all $i \in \N$ we have $K' \rtimes_\p i\ell\Z  \subg \f G$.
\end{lem}

\begin{proof}
Suppose $|K:K'| = n$. Let $r$ be the number of subgroups in $K$ with index $n$, this is finite by \cref{lem:finitely.many.subgroups.index}. Note that $\p(K')$ is a finite index subgroup of $K$ with index $n$, therefore $\p$ permutes the subgroups of $K$ with index $ n$. Therefore,  we have $\p^{r!}(K')  =K'$. By \cref{subgroup_of_semidireict_product}, we have  $K' \rtimes_\p r!\Z  \subg\f G$.
\end{proof}

\subsection{Some  linear algebra results}
Let $\F$ be a field. Throughout this section, we will be working with the finite-dimensional vector space $V = \F^n$ over the field $\F$.

\begin{lem} \label{commue_share_an_eigenvector} \label{share_a_last_eigenvector}

 Let $A_1, \ldots, A_m \in \GL(n, \F)$	 be a set of matrices that pairwise commute. Suppose that $\ch(A_1), \ldots, \ch(A_m)$ split over $\F$. Let $\l_1$ be an eigenvalue of $A_1$. Then, 
\begin{enumerate} [(i)]
\item there exists $\l_2, \ldots, \l_m$  such that $ \bigcap_{i=1} ^{n} E_{\l_i}(A_i)  \neq \{0\}.$

\item It follows that $A_1, \ldots,A_m \in \GL(n, \F)$ share a last generalised eigenvector in  $\GE_{\l_1}(A_1)$.

\end{enumerate}

\end{lem}

\begin{proof}
\mbox{}
\begin{enumerate} [(i)]
\item

We will prove the statement via induction on $m$.
The case for $m=1$ is given in the assumption. Suppose the statement holds for $m-1$. i.e. there exists $\l_2, \ldots,\l_{m-1}$,  such that $ \bigcap_{i=1} ^{m-1} E_{\l_i}(A_i)  \neq \{0\}$. Recall the fact that  commuting linear map preserve each other's eigenspaces, it follows that $$A_m \left( \bigcap_{i=1} ^{m-1} E_{\l_i}(A_i) \right) =   \bigcap_{i=1} ^{m-1} A_m \left(E_{\l_i}(A_i) \right) \subseteq \bigcap_{i=1} ^{m-1} E_{\l_i}(A_i).$$ Hence, \cref{eigenvector_in_invariant_subspace} tells us that $A_m$ has an eigenvalue $\l_m$ with an eigenvector in $\bigcap_{i=1} ^{m-1} E_{\l_i}(A_i)$, which is to say that $ \bigcap_{i=1} ^{m} E_{\l_i}(A_i)  \neq \{0\}$.

\item

We will prove the statement by induction on $n = \dim \GE_{\l_1}(A_1)$. For the base case, suppose  $\GE_{\l}(A_1) =\langle v \rangle$.
By (i), 
 $v$ is an eigenvector for  $A_1, \ldots,A_m$. It is also a  last generalised eigenvector of  $A_1$ by definition. Suppose $A_1$ has $k$ distinct eigenvalues $ \a_1= \l_1, \a_2 \ldots,\a_k$. Then we can decompose $\F^n =  \GE_{\a_1} \oplus _{j=2} ^k \GE_{\a_j}.$ Define $W =  \oplus _{j=2} ^k \GE_{\a_j}$. Then for each $i$ and $j$,  $A_i$ preserves $\GE_{\a_j}$. Hence $A_i - \l_i I$ preserves  $\GE_{\a_j}$. Therefore, $A_i - \l_i I$ preserves $W$. Let $y \in \F^n$, then $y = \a v + w$ for some $\a$ in $\F$ and $w$ in $W$. We have $(A_i  - \l_i I)y = (A_i  - \l_i I)w \in W.$ It follows that $v$ is an last eigenvalue of $A_i$ with respect to $\l_i$.

We now proceed to the induction step. Suppose  $\dim \GE_{\l_1}(A_1) \geq 2$. According to \cref{commue_share_an_eigenvector},  $A_1, \ldots,A_m$ share an ordinary eigenvectors in $v \in E_{\l_1}(A_1)$. Consider the  map
\begin{align*}
\bar{A}_i :  \frac{\F^n}{\langle v \rangle}  &  \longrightarrow \frac{\F^n}{\langle v \rangle} \\
 \bar{u} \coloneqq u+\langle v \rangle &\mapsto A_i u + \langle v \rangle,
\end{align*}
where we  use the notation $x \mapsto \bar{x}$ for the quotient map from 
$\F^n$ to 
$\F^n/\langle v \rangle$. By some computation, we can see that 
\begin{enumerate} [(1)]
\item $\bar{A}_i$ and $\bar{A}_j$ pairwise commute; 
\item $\ch (A_i)$ splits over $\F$ for all $i$; 
\item $\GE_{\l_1}(\bar{A}_i)  \neq \{0\}$. .
\end{enumerate}
By induction hypothesis, $\bar{A}_i$ shares a last generalised eigenvector $\bar{w} \in \GE_{\l_1}(\bar{A}_1)$. We claim that the vector $w$ belongs to $\GE_{\l_1}(A_1)$ and it is a last eigenvector for all $A_i$'s. To see it is a generalised eigenvector: since $\bar{w} \in \GE_{\l_1}(\bar{A}_1)$, there exists $r$, such that $(A_1 - \l_1 I )^r w \in \langle v \rangle$,  i.e. $(A_1 - \l_1 I )^r w = \b v$ for some $\b \in \F$.
Therefore, $(A_1 - \l_1 I )^{r+1} w = (A_1 - \l_1 I ) \b v  = 0$, i.e. $w \in \GE_{\l_1}(A_1)$. To see it is a last eigenvector, 
 let  $u \in  \GE_\l(A_i)$, such that $(A_i- \l_i I)u = w$. Then we have $\bar{u} \in  \GE_\l(\bar{A_i})$ such that $(\bar{A_i}- \l_i I)\bar{u} = \bar{w}$, this contradicts the fact that $\bar{w}$ is a last eigenvector. 
 \qedhere
 \end{enumerate}
 
\end{proof}
In general, the commutativity of matrix multiplication is not transitive. However, the transitivity holds when we have the following restrictions on the matrices.

\begin{lem}  \label{commutivity_is_transitive_if}
Let $A$, $B$, $C \in \GL(n,\Z)$ such that $A$ commute with both $B$ and $C$. Suppose $A$ has an eigenvalue $\a$ with $|\a| \neq 1$. Then

\begin{center}
$A^rBC = CB$ for some $r\in \Z \Rightarrow r=0$.
\end{center}
\end{lem}

\begin{proof}

We first claim that if $\b$ is an eigenvalue  of $B$ with $ E_\a(A) \cap E_\b(B) \neq \{0\}$, then $ E_\a(A) \cap E_{\frac{\b}{\a ^r}} (B) \neq \{0\}.$ To see this, let $v$ be a non-zero vector in $ E_\a(A) \cap E_\b(B)$. We have 
\begin{align}
A^rBC v &= CB v \nonumber\\
 B  C (\a ^r v) &= C (\b v)   \qquad \text{since $A$ commute with both $B$ and $C$,} \nonumber \\
 B  (C  v) &= \frac{\b}{\a ^r} C  v.   \label{first}
\end{align}
Hence, $ C  v \in E_{\frac{\b}{\a ^r}} (B)$. Again, since $A$ commutes with $C$, we have  $A C  v  = \a  C v$  and hence $C  v \in E_\a(A)$, which concludes our proof for the claim. Finally, since $A$ and $B$ commute, \cref{commue_share_an_eigenvector}(i) tells us that $B$ has an eigenvalue $\b$ such that $ E_\a(A) \cap E_\b(B) \neq \{0\}$. It follows that every element in $\{ \frac{\b}{\a^{ir}} | i\in \N \}$ is an eigenvalue of $B$. Therefore, we must have $r=0$. 
\end{proof}

\begin{lem} \label{auto_with_all_eigenvalue_1}
Let $T$ be a $\Z$-module automorphism of  $\Z^m$ with all eigenvalue value $1$.  
Then there exists a free basis $\mathcal{B} \coloneqq \{v_1, \ldots, v_m\}$ of $\Z^m$, such that $[T]_\mathcal{B}$, the matrix corresponds to $T$ with respect to $\mathcal{B}$ is an upper triangular matrix . 
\end{lem}
Note that all the diagonal entries of $[T]_\mathcal{B}$ are $1$.

\begin{proof}
We will prove this by induction on $m$. For $m=1$, the statement is clear. Suppose the statement holds for $m-1$. Let $T$ be a  linear map of  $\Z^m$ with all eigenvalue value $1$. Since $T$ has an eigenvalue $1$, there exists $v_1 = (z_1, z_2, \ldots, z_m) \in \Z^m$ with $\gcd(z_1, z_2, \ldots, z_m) =1$ such that $Tv_1 = v_1$. Then $T$ induce a  linear map $\bar{T}$ on $ \Z^m / \langle v_1  \rangle \iso \Z^{m-1}$. Also, all the eigenvalues of $\bar{T}$ are $1$. By induction hypothesis,    $\Z^m / \langle v_1 \rangle$ has a basis $\mathcal{B'} = \{ v_2 +\langle v_1  \rangle, \ldots,  v_m+\langle v_1  \rangle \}$, such that $\bar{T} _\mathcal{B'}$ is upper triangular. Take $\mathcal{B} = \{ v_1, v_2 , \ldots,  v_m  \}$, we can see that $[T]_\mathcal{B}$ is upper triangular. 
\end{proof}

Here are some useful equivalent characterization of matrices in $\GL(m,\Z)$ from \citep[Section 4]{murota_linear_2022}.

\begin{thm} \label{change_of_basis_matrix}
Let $M$ be an $m$ by $m$ integer matrix, then the following are equivalent.
\begin{enumerate} [(i)]
\item The column of $M$ form a free basis of $\Z^m$.
\item $M$ is invertible and $M\i$ has entries in $\Z$.
\item $det(M) = \pm 1$.
\end{enumerate}
\end{thm}

\begin{lem} \label{all_eigenvalue_1_similar_to_upper_triangular_matrix}
Let $M \in \GL(m,\Z)$ with all eigenvalues $1$. Then $M$ is similar to an upper triangular $A$ matrix via an integral matrix.
\end{lem}

\begin{proof}
This follows from \cref{auto_with_all_eigenvalue_1} and \cref{change_of_basis_matrix}.
\end{proof}

\begin{lem}\label{matrix_form}
Let $Q \in \GL(m,\Z)$ with all eigenvalues $1$. Then $Q$ is similar to an upper triangular matrix $A$ via an integral matrix, where $A = (a_{i,j})$ satisfies the following property: for every $  l \geq 3$, if $a_{i,j} = 0$ for all $1 <i < j < l$, then $a_{2,l} = a_{3,l} = \ldots =a_{l-1,l} = 0$. Equivalently, for $l \in [q-1]$, if $col_j(A) = e_j$ for all  $j \in [l]$, then $col_{l+1} (A) = e_{l+1} + q e_1$ for some $q \in \Z$:
\end{lem}

\begin{equation*}
\begin{pNiceMatrix}
 1 & 0 & 0 & \cdots & 0 & q & * & \cdots & * \\
 & 1 & 0 & \cdots & 0 & 0 & * & \cdots & * \\
 &  & 1 & \cdots & 0 & 0 & * & \cdots & * \\
 &  &  & \ddots & \vdots & \vdots & \vdots & \ddots & \vdots \\
 &  &  &  & 1 & 0 & * & \cdots & * \\
 &  &  &  &  & 1 & * & \cdots & * \\
 &  &  &  &  &  & 1 & \ddots & \vdots \\
 &  &  &  &  &  &  & \ddots & * \\
 &  &  &  &  &  &  &  & 1
\end{pNiceMatrix}
\end{equation*}

\begin{proof}

By \cref{all_eigenvalue_1_similar_to_upper_triangular_matrix}, we can assume $Q =(q_{i,j})$ is an upper-triangular matrix.  Let $l \geq 3$, such that $q_{i,j} = 0$ for all $1 <i < j < l$.  Given $r,s \in [l]$, with $r \neq s$ and $ \l \in \Z$, define $E(\l,r,s)$ to be the $m$-by-$m$ matrix equal to the identity with an additional entry $\l$ in the $(r,s)$ position. Note that $E(\l,r,s)\i =E(-\l,r,s)$.
Defining $Q' = E(\l,r,s) Q E(-\l,r,s)$, we see that $Q' = (q'_{i,j})$ is an upper-triangular matrix such that 
\begin{equation*}
    q'_{i,j}= 
    \begin{cases}
      0 & \text{if }  1 <i < j < l \\
      q_{r,l}+\l q_{s,l} & \text{if } i = r \text{ and } j = l \\
      q_{i,j} & \text{if } i \neq r \text{ and } j = l \\
    \end{cases}
  \end{equation*}
It follows that $Q$ is similar to a matrix that satisfies the properties in the lemma. \qedhere

\end{proof}

The following well-known result characterises the subgroups of finitely generated free abelian groups.

\begin{lem}[{\citep[Theorem 1.6]{algebra_Thomas}}] \label{subgroup_of_zn}
Let $H$ be a subgroup of $\Z^n$ with rank $r$.
Then there exists a free basis  $y_1, \ldots y_n$ of $\Z^n$ and $d_1, \ldots,  d_r \in \N$, such that  $H = \langle d_1y_1,  \ldots, d_r y_r  \rangle.$ 
\end{lem}

\begin{cor} \label{torsion_free_quotient}
For any $M \in \GL(m,\Z)$, $\frac{\Z^n}{E_1(M)}$ is free abelian. 
\end{cor}

\begin{proof}
Suppose $E_1(M) \subg \Z^n$ has rank $r$. \cref{subgroup_of_zn} tells us that there exists a free basis  $y_1, \ldots y_n$ of $\Z^n$ and $d_1, \ldots,  d_r \in \N$, such that $E_1(M) = \langle d_1y_1,  \ldots, d_r y_r \rangle.$ Given any $i \in [r]$, we have $d_i M(y_i) = M(d_iy_i) = d_i y_i$. Hence $M(y_i) =y_i$, i.e. $y_i \in E_1(M)$. Hence, we must have $d_i = 1$. Therefore, $\frac{\Z^n}{E_1(M)} \iso \langle  y_{r+1},  \ldots,  y_n \rangle $ is torsion free.
\end{proof}

\begin{lem} \label{product_of_n_nilpootent_matrices}
Let $Q_1, \ldots, Q_n \in \F^{n \times n}$ be a set of pairwise commuting matrices with all eigenvalues $0$, then $\prod _{i=1} ^n Q_i =0$.
\end{lem}

\begin{proof}
The result follows immaculately from this standard fact: Let $S$ and $T$ be two commuting linear maps on $V$. Suppose all eigenvalues of $S$ are $0$ and $T$ is a non-zero linear map, then $\ker (T) \subsetneqq \ker(ST)$.
\end{proof}

\subsection{Properties of $P$ when $\FP$ is free abelian and $\PFP$ is torsion free} \label{Properties_P_is_free_abelian}

Let $P$ be a poly-$\Z$ group with Hirsch length $h$, suppose that $\FP \iso \Z^n $ and $h(\PFP) =m=h-n$. 
 With a slight abuse of notation, we omit the isomorphism map and write $\v \in (\Z^n, +)$, to mean its corresponding elements $v \in (\FP,\cdot)$. Every automorphism $\theta$ of $\FP$ corresponds to a matrix $M \in \glnz$, and for $v \in \FP$, we will often write $\theta(v) = M(\v)$.  Furthermore, given a subgroup $V \subg \FP$, we will denote by $\V$ the corresponding subgroup in $\Z^n$. Suppose in addition that $\PFP$ is torsion free.  According to \cref{write_Polyz_group_with_torsion_free_quotient_as_semi_direct_product}, we have 
\begin{equation} \label{label_for_torsion_free_polycyclic}
P \iso \FP \rtimes _{\p_1} \langle t_1\rangle \rtimes_{\p_2} \langle t_2 \rangle\rtimes \ldots \rtimes _{\p_{m}} \langle t_m\rangle 
\end{equation}
where each $t_i$ generates a copy of $\Z$, and each $\p_i$  is an automorphism for $\FP \rtimes _{\p_1} \langle t_1\rangle \rtimes_{\p_2} \langle t_2 \rangle\rtimes \ldots \rtimes _{\p_{i-1}} \langle t_{i-1}\rangle$. For $i =0, 1,2, \ldots, m$ we will define $K_i \coloneqq \FP \rtimes \langle t_1\rangle \rtimes \langle t_2\rangle\rtimes \ldots \rtimes\langle t_i\rangle.$ By \cref{normal_subgroup_contains_fit}, it follows that  
\begin{equation} \label{fitting_of_K_i}
\FP = \Fit(K_0) \subset \Fit(K_1) \subset \ldots \subset \Fit(K_m) = \FP.
\end{equation}
Hence, for each $i$ we have $\FP=\Fit(K_i) \csubg K_i$.  Therefore, each $\phi_i$  preserves $ \Fit (P) $, so it induces an automorphism on $\FP \iso \Z^n$, which corresponds to a matrix $M_i \in \GL(n,\Z)$.

\begin{lem}  \label{same_coset_M_and_Mi_commute}
Suppose $\FP$ is free abelian and $\PFP$ is torsion free.  Let $\p \in \Aut (P)$. Suppose $i \in[m]$ is such that  $\p(t_i) \in \FP t_i$, i.e. $\p(t_i) = u t_i$ for some $u \in \FP$.
Let $M$ be the matrix corresponding to $\p |_{\FP}$. Then $M$ commutes with $M_i$.
\end{lem}

\begin{proof}
Since $\p$  is an automorphism of $P$, for each $v \in \FP$, we have $M M_i \v=\p ( M_i \v)=\p (t_i v t_i\i)=\p (t_i) \p (v) \p (t_i)\i  =  u t_i (M \v) t_i\i u\i = t_i (M \v) t_i\i    =M_i M \v$.
\end{proof}

\begin{cor} \label{commute}
Suppose both $\FP$ and $\PFP$ are free abelian. Then

\begin{enumerate} [(i)]
\item the matrices $M_i$'s pairwise commute;
\item for every $i \in [m]$, we have $E_1(M_i) \nsubg P$.
\end{enumerate}

\end{cor}

\begin{proof}
(i) follows immediately from \cref{same_coset_M_and_Mi_commute}. We will prove (ii). Let $\v\in E_1(M_i)$.  Given $g\in P$, we can write $g = u t_1^{\alpha_1} t_2^{\alpha_2} \ldots t_m^{\alpha_m}$  for some $u \in \FP$, and $\a_i \in \Z$.  It is clear that  $gvg\i = M_1^{\alpha_1} M_2^{\alpha_2} \ldots M_m^{\alpha_m} \v.$
By (i), we have $M_i( M_1^{\alpha_1} M_2^{\alpha_2} \ldots M_m^{\alpha_m} \v) = M_1^{\alpha_1} M_2^{\alpha_2} \ldots M_m^{\alpha_m} (M_i\v) = M_1^{\alpha_1} M_2^{\alpha_2} \ldots M_m^{\alpha_m} \v.$ Therefore, $M_1^{\alpha_1} M_2^{\alpha_2} \ldots M_m^{\alpha_m} \v\in E_1(M_i)$. 
\end{proof}

The following theorem is a special case of \citep[Exercise~7 on p.~92]{polycyclic}.

\begin{thm} \label{v_nilp_implies_v_abelian}
Suppose that $\FP$ is torsion-free abelian and $\PFP$ is torsion-free virtually nilpotent. Then $P$ has a finite-index normal subgroup $P'$ such that $\FPP = \FP$ and $P'/\Fit(P')$ is abelian.

\end{thm}

\begin{proof}
We will prove by induction on $i$ that for each $i = 1, \ldots m$ there exists $K_i' \nsubg \f K_i$ with $\FKiP = \FP$ such that $K_i'/\FKiP \iso \Z^i$, the case $i=m$ being the statement of the theorem.
For $i=1$, it is clear as $K_1 = \FP \rtimes \Z$.
Now, suppose the claim holds for $i$ and let $K'_i$ be the subgroup that satisfies the assumptions in the induction hypotheses. Given an automorphism $\p$ on $K_{i}'$, we will denote ${\bar{{ \p}}}$ to be the induced  automorphism on $K_{i}'/\FKiP$.  By \cref{big_enough_l_K_lZ_is_a_subgroup}, there exists $\ell\in\N$, such that $K_i' \rtimes_{ \p_{i+1}} \ell\Z_{(i+1)} \subg K_{i+1}$.  By the induction hypothesis,  $\FP = \Fit(K_i') \csubg K_i' \nsubg K_i' \rtimes_{ \p_i} \ell\Z_{(i+1)}$, so $\FP \nsubg K_i' \rtimes_{ \p_i} \ell\Z_{(i+1)}$.   Since $\PFP$ is virtually nilpotent, the subgroup 
$   \frac{K_i' \rtimes_{ \p_i}\ell\Z_{(i+1)}}{\FP} \subg \PFP$
is also virtually nilpotent. Note that
\begin{equation*} 
\frac{K_i' \rtimes_{ \p_i}\ell\Z_{(i+1)}}{\FP} \cong \frac{K_{i}'}{\FP}\rt_{\bar{{ \p_i}} } \ell\Z \cong \Z^i  \rt _{\bar{{ \p_i}}} \ell  \Z \iso \Z^i  \rt_{\bar{{ \p_i}}^\ell}  \Z.
\end{equation*}
Let $Q_{-} \in \GL(i,\Z)$ be the matrix corresponding to ${\bar{{ \p_i}}^l}$. By \cref{classify_Zn_by_Z}, all eigenvalue of $Q_{-}$ have absolute value $1$. \ Therefore, there exists $k\in\N$ such that all eigenvalues of ${Q_{-}}^{k}$ are $1$. Define $K_{i+1}' =  K_i' \rt _{ \p_{i+1}} \ell k \Z \iso K_i'\rtimes _{{ \p_i}^{\ell k}} \Z$, and let $\sigma = { \p_i}^{\ell k}$. 
Then $\KF{i}$ has a basis $\{\FP x_1, \FP x_2, \ldots,  \FP x_i\}$ such that the matrix $Q$ corresponding to $\bar{\sigma}$ is an upper-triangular matrix with the form described in \cref{matrix_form}. Since $K_i'/\FKiP$ is abelian,  $K_i'$ has a subnormal series
$$
\{1\} \nsubg \FP \nsubg \langle \FP, x_1 \rangle \nsubg 
 \ldots \nsubg \langle \FP, x_1, x_2, \ldots x_i \rangle = K_i'.
$$

Note that all the successive quotients are isomorphic to $\Z$. It follows  by induction that 
$$K_i' \iso \FP \rtimes _{\sigma_1} \langle x_1\rangle \rtimes_{\sigma_2} \langle x_2 \rangle\rtimes \ldots \rtimes _{\sigma_{i}} \langle x_i\rangle,$$
where each $\sigma_j$  is an automorphism for $\langle \FP, x_1, x_2, \ldots x_{j-1}  \rangle$. 
 According to \eqref{fitting_of_K_i}, each $\sigma_j$  preserves $\Fit (P)$, hence it induces an automorphism on $\FP \iso \Z^n$, which corresponds to a matrix $M_j \in \GL(n,\Z)$.   Let $M   \in \GL(n,\Z)$ be the matrix that corresponds to the automorphism induced by $\sigma$ on $\FP$.
 We claim that $Q$ is the identity matrix. We will prove by induction that for each $k \in [i]$, we have $col_k(Q)= e_k$,  this in particular implies that  $M$ commutes with $M_k$ by \cref{same_coset_M_and_Mi_commute}.  Clearly, $col_1(Q)= e_1$. Suppose the statement holds for the first $k-1$ columns.  By the construction of $Q$, 
we have $Q(\FP x_k)=   \FP x_k+ \FP q x_1=  \FP(x_1^q x_k)$, so  $\sigma (x_k)=u_k x_1^q x_k$ for some $u_k \in \FP$. Since $\sigma$ is an automorphism on $K_i'$, for all $v \in \FP$ we have $
M M_k \v=\sigma( M_k \v)=\sigma(x_k v x_k\i)=\sigma(x_k) \sigma(\v) \sigma(x_k)\i  
= u_k x_1^q x_k  (M \v) (u_k x_1^q x_k)\i    = M_1^q M_k M \v$, so we get $M M_k  = M_1^q M_k M$.  We also have that $M_1M_k =M_kM_1$ by \cref{commute}, and $M M_1  = M_1 M$ by the base case. Since  $\FP \rtimes _{\sigma_1} \langle x_1\rangle $ is not virtually nilpotent, \cref{classify_Zn_by_Z} tells us that $M_1$ has an eigenvalue with absolute value not equal to $1$, so $q=0$ by \cref{commutivity_is_transitive_if} and $col_k(Q)= e_k$ as claimed. Therefore, $Q$  is indeed the identity matrix, which implies the $\bar{\sigma}$, the automorphism induced by  $\sigma = { \p_i}^{\ell k}$ on the quotient $K_i'/\FKiP$, is the identity automorphism. 
Finally, the definition of $K_{i+1}'$ and \cref{normal_subgroup_contains_fit} tell us that  $\FP = \Fit(K_i') = \Fit(K_{i+1}')$, hence  $
 K_{i+1}'/\Fit(K_{i+1}')  \iso  (K_i'/\Fit(K_{i}')) \rtimes _{\bar{ \p_i}^{\ell k}} \Z \iso \Z^{i+1}$, which concludes the proof. \qedhere

\end{proof}

\subsection{Groups of the form $P=\Z^n \rtimes \Z^m$}
In this subsection, we will be looking at a class of special  poly-$\Z$ groups of the form $P=\Z^n \rtimes \Z^m$, we will be using the following labelling: 
\begin{align} 
P &\iso \Z^n \rtimes_\Phi \langle t_1, t_2, \ldots t_m \rangle  \label{labelling_zn_by_zm}\\
&\iso  \left\langle 
       \begin{array}{l|cl}
            \FP,            & \FP \iso \Z^n ; \quad t_i v t_i\i = M_i \v \text{ for every } v \in \FP; \\
            t_1, \ldots t_m  & t_i t_j t_i \i =  t_j \text{ for } i>j            \label{relation_zn_by_zm}                                  
        \end{array}
     \right\rangle
\end{align}
where $\Phi$ is a homomorphism from  $\Z^m$  to $\Aut(\Z^n) = \GL(n, \Z)$. The map $\Phi$ is entirely determined by the automorphisms $\phi_i \coloneqq \Phi(t_i)$ where $i \in [m]$, each $\phi_i$ corresponds to  a matrix $M_i \in \GL(n,\Z)$. According to \cref{commute}, the matrices $M_i$'s pairwise commute.

\begin{lem} \label{commutator}
Let $G$ be group and $H \subg G$ be a subgroup generated by $T$.
Suppose $H \subeq C_G([H,G])$, then $ [H,G] = \langle [t,g] \mid t \in T, g \in G\rangle.$
\end{lem}

\begin{proof}
This follows from a commutator law: for all $x, y ,z \in G$, $[xz,y] = x[z,y]x \i [x,y]$.
\end{proof}

\begin{lem} \label{lower_central_series_of_zn_zm}
Suppose $P=\Z^n \rtimes \Z^m$. Let $X=\{I-  M_1^{\alpha_1} M_2^{\alpha_2} \ldots M_m^{\alpha_m} \mid \a_j \in \Z\}\subeq \Z^{n\times n}$. Denote $P_i$ the $i$-th term in the lower central series of $P$. 
Then  \
\begin{enumerate} [(i)]
\item for every $i \geq 1$, $P_{i} \subeq C_P ( [P_{i},P])$;
\item for every $i \geq 1$, $P_{i} =  \langle X^i \Z^n \rangle = \langle Q \v \mid Q \in X^i, \v \in \Z^n \rangle$   where $X^i$ is the sets of products of $i$ matrices in $X$;
\item if for every $i\ in [m]$ all the eigenvalues of the matrix $M_i$ are $1$ then $P$ is nilpotent; in particular, if all the eigenvalues of the every $M_i$ have absolute value $1$, then $P$ is virtually nilpotent.
\end{enumerate}
\end{lem}

\begin{proof}
(i) follows from the fact that both $P_{i} $ and $ [P_{i},P]$ are subgroups of $\Zn$ . We will prove (ii) via induction on $i$.  Given $g\in P$, we can write $g = u t_1^{\alpha_1} t_2^{\alpha_2} \ldots t_m^{\alpha_m}$  for some $u \in \Z^n$, and $\a_j \in \Z$.  Suppose $i = 1$. Given $ h \in P$, we can write $h = v t_1^{\beta_1} t_2^{\beta_2} \ldots t_m^{\beta_m}$  for some $v \in \Z^n$, and $\b_j \in \Z$. Then we have  
\begin{align*} 
hgh\i g\i &= ( v t_1^{\beta_1} t_2^{\beta_2} \ldots t_m^{\beta_m} )( u t_1^{\alpha_1} t_2^{\alpha_2} \ldots t_m^{\alpha_m} ) (v t_1^{\beta_1} t_2^{\beta_2} \ldots t_m^{\beta_m} )\i   (u t_1^{\alpha_1} t_2^{\alpha_2} \ldots t_m^{\alpha_m})\i \\
&=  v \bigg(t_1^{\beta_1} t_2^{\beta_2} \ldots t_m^{\beta_m}  u (t_1^{\beta_1} t_2^{\beta_2} \ldots t_m^{\beta_m} )\i  \bigg) \bigg(  t_1^{\alpha_1} t_2^{\alpha_2} \ldots t_m^{\alpha_m}  v\i    ( t_1^{\alpha_1} t_2^{\alpha_2} \ldots t_m^{\alpha_m})\i \bigg) u\i \\
&=\v + M_1^{\b_1} M_2^{\b_2} \ldots M_m^{\b_m} \u  - M_1^{\alpha_1} M_2^{\alpha_2} \ldots M_m^{\alpha_m} \v -\u \\
&=(I-M_1^{\alpha_1} M_2^{\alpha_2} \ldots M_m^{\alpha_m}) \v - (I-M_1^{\b_1} M_2^{\b_2} \ldots M_m^{\b_m}) \u \in \langle X^1 \Z^n \rangle.
\end{align*}

For the induction step, suppose the statement holds for $i$. Then by part (i) and \cref{commutator}, we see that  
$\Pp = [P_{i},P] =  \langle [v,g] \mid \v \in   X^i \Z^n, g\in P \rangle.$
For every $\v \in X^i \Z^n$,  we have  \[ [v,g] =vgv\i g\i = (I- M_1^{\alpha_1} M_2^{\alpha_2} \ldots M_m^{\alpha_m}) \v \in  X^{i+1} \Z^n , \]  
which proves (ii). Finally, (iii) follows from (ii) and \cref{product_of_n_nilpootent_matrices}.
\end{proof}

 Our next goal is to   to characterise $\Z^n \rtimes \Z^m$ groups that have   $\alpha$-almost flat coset spaces.  

\begin{lem} \label{mult_char_polys}
Suppose $f_1(x), \ldots, f_m(x)$ are non-zero monic polynomials in $\Z[X]$, then there are infinitely many primes $p$ such that every $f_i(x)$ split over $\F_p$ with no zero roots.
\end{lem}

\begin{proof}
This follows from \cref{split_over_inf_primes_with_no_zero_roots} and \cref{product_of_poly}.
\end{proof}

The following results help us to compare the diameter of a certain subgroup of the semidirect product to the whole group.

\begin{lem} \label{same_diam}
Let $G = K \rt_\p \Z = \langle T  , t^{\pm1} \rangle$ where $K =  \langle T \rangle$ and $\Z $ is the infinite cyclic group generated by $t$. 
Let $K' \subg K$, such that $\p(K') = K'$, in particular,  $K' \rt_\p \Z \subg G$. Then, $\diam_ {\{  T  , t \}} G/ (K' \rt_\p \Z) \leq  \diam_T (K/K').$
\end{lem}

\begin{proof}
Let $g = G$, then $g = k z$ for some $k \in K$ and $z\in \Z$. 
Then $k = t_1 \ldots t_m k'$ for some $m  \leq \diam_T (K/K')$, $t_i \in T$ and $k' \in K'$.
Hence $g = t_1 \ldots t_m (k' z ) \in \{  T  , t \} ^m  (K' \rt_\p \Z)$.
\end{proof}

\begin{lem} \label{characteriseing_zm_zn}
Let $\alpha\in(0,1]$, and suppose $P=\Z^n \rtimes \Z^m$ has uniformly $\alpha$-almost flat coset spaces. Then $P$ is virtually nilpotent.
\end{lem}

\begin{proof}
Suppose $P$ is not virtually nilpotent. \cref{lower_central_series_of_zn_zm} (iii) tells us that there exists a matrix $M_i$ that has an eigenvalue that does not have absolute value $1$. Without loss of generality, we will let $M_1$ to be this matrix. We will use the natural generating set $\{\pm e_i, t_j ^{\pm 1} \mid i \in [n], j\in [m]\}$ for $P$. Fix $\a \in (0,1]$, we want to show that $P$ does not have uniformly   $\alpha$-almost flat coset spaces. 
By \cref{mult_char_polys}, there exists an infinite set $\cp$ of primes  such that every $\ch(M_i)$ splits over $\Fp$ with no zero roots. \
Using the notation we introduced in \cref{subsection_uniform_d_inzn_by_z}, for each $p \in \cp$, let $\tilde{\l}_p$ be a eigenvalue of $\tilde{M}_1$ with the largest multiplicative order over $\F_p$. According to \cref{share_a_last_eigenvector} (ii) and \cref{commute}, $\ti{M}_1, \ldots \ti{M}_m \in \GL(n, \Z_p)$ share a last generalised eigenvector $\tilde{v}_p$ in  $\GE			_{\tilde{\l}_p}(\tilde{M_1})$.
Follow the notation in \cref{set_of_subgroups}, let $G= \Z^n \rt \la t_1\ra$ with generating set  $\{ \pm e_i, t_1^{\pm 1} \mid i \in [n]\}$; let $A_p = A_{v_p}$ be a subgroup of $\Z^n$ of index $p$ that satisfies the properties in  \cref{index_p_subgroup}. It follows that $A_p$ is preserved by all $M_i$'s. 
Let $\H_\cp$ be the set of subgroups of $\Z^n$ that satisfies the assumption in \cref{set_of_subgroups}, so there exists $H \in \H_\cp$ such that $\diam(G/H) < |G:H|^\alpha$. 
Since $P = \Z^n \rt \Z^m$, it follows from \cref{subgroup_of_semidireict_product} that $H' = H   \rtimes\langle t_2, \ldots t_m \rangle  \iso H \rtimes \Z_{(2)}\rtimes \ldots \rtimes\Z_{(m)} $ is a subgroup of $P$. Also, by  \cref{same_diam} and induction, we have $\diam(P/H') \leq \diam (G/H) < |G:H|^\alpha = |P:H'|^\alpha$.
\end{proof}

For the rest of this subsection, we  will continue to use the labelling for $P=\Z^n \rtimes \Z^m$ defined in \eqref{labelling_zn_by_zm}. In addition, we will assume $\FP= \Z^n$  and  $1$ is not an eigenvalue of $M_1$.  After fixing a labelling for $P$, we can define the following notations:

\begin{enumerate} [(i)]
\item $\A(P) = \{  \p \in \Aut(P) \mid \p(t_i) \in  \Z^n t_i\text{ for each } i \};$
\item $U(P) \coloneqq \{u \in \FP \mid \text{ for every } j, (I-M_j)(I-M_1)\i \u \in \Z^n  \};$
\item $C(P) = \bigcap_{j=1}^{m}C_{\GL (n,\Z)} (M_j) \subeq \GL (n,\Z). $
\end{enumerate}
We will show later that $\A(P) \subg Aut(P)$ and $\A (P) \iso U(P) \rt C(P)$.  Note that the definitions of the above three subgroups are dependent on the labelling of the group $P$. We can think of $\A(P)$ as the sets of automorphisms of $P$ that induces identity map on $P /  \Z^n$. In particular,  $\Inn(P)\le \Aut^+(P)$.

\

\begin{lem} \label{what_elt_in_A_plus_looks_like}
Let $\p \in \A(P)$. 
\begin{enumerate} [(i)]
\item Let $M \in \GL(n,\Z)$ be the matrix corresponding  to the restriction of $\p$ on $\Z^n$. \ Then $M \in  C(P)$.
\item \label{item:t_i.det.by.t_1}We have $\p(t_1) = u t_1$ for some $u \in U(P)$, and for each $i$, $\p(t_i) = (I-M_i)(I-M_1)\i \u t_i$.
In particular,  each $\p(t_i)$ is determined by $\p(t_1)$. 

\end{enumerate}

\end{lem}

\begin{proof}
(i) can be proved using a similar idea in \cref{same_coset_M_and_Mi_commute}, we will prove (ii). For each $i$, we have $\p(t_i)  = u_i t_i$ for some $u_i \in \Z^n$.   Therefore
\begin{align*}
u t_1 &=\p( t_1 )= \p(t_i t_1 t_i\i)=\p(t_i) \p(t_1) \p(t_i)\i  = u_i t_i u (t_1 t_i\i) u_i\i=  \\
&  = u_i t_i u ( t_i\i t_1) u_i\i= u_i (M_i \u) t_1  u_i\i = u_i (M_i \u)  (t_1  u_i\i) = (\u_i +  M_i \u  -M_1 \u_i)t_1.  
\end{align*}  
Since $M_1$ does not have eigenvalue $1$,   $(I-M_1)$ is invertible. Hence, $\u_i = (I-M_i)(I-M_1)\i \u$. 
\end{proof}

\begin{lem}
We have $\A(P) \subg Aut(P)$.
\end{lem}

\begin{proof}
Clearly $\id \in A$.  We next check the closure property: let $ \phi_1$, $\p_2 \in \A(P)$, then for each $i$, $\p_1(t_i)=  u_i t_i$, $\p_2(t_i)=  v_i t_i$, for some $u_i$, $v_i \in \Z^n$.  Then  $\p_2 \p_1 (t_i) = \p_2 (u_i t_i)=\p_2 (u_i)\p_2( t_i) = \p_2 (u_i) v_i t_i \in  \Z^n t_i.$ Hence, $\p_2 \p_1 \in \A(P)$. Lastly we check the existence of inverse, we have $\p_1 \i (t_i) = \p_1 \i (u_i\i) t_i \in  \Z^n t_i.$    Hence, $\p_1 \i \in \A(P)$.
\end{proof}

\begin{lem} \label{existence_of_auto}
There is a natural way to identify each $u \in U(P)$ and each $M \in C(P)$ with an automorphism in  $\A(P)$ as follows:

\begin{enumerate} [(i)]
\item given $u \in U(P)$, there exists a unique  $\a_u  \in \A(P)$ such that $\a_u(t_i) = (I-M_i)(I-M_1)\i \u t_i$ and $\a_u(v)=v$ for all $v \in \Z^n$;

\item \label{matrix_auto} given $M \in C(P)$, there exists a unique $\g_M \in \A(P)$ such that  $\g_M(t_i) = t_i$ for all $i$ and $\g_M(v)= Mv$ for all $v \in \Z^n$.
\end{enumerate}
\end{lem}
\begin{proof}
Recall that each polycyclic group is residually finite and therefore Hopfian, which implies that any surjective homomorphism from a polycyclic group to itself is actually an automorphism. Hence it is enough to check the maps $\a_u $  and $\g_M$ extends to a well-defined surjective homomorphism from $P$ to itself. We need to  check that they preserve the relations of the group $P$ defined in \eqref{relation_zn_by_zm}.
Let $w, v \in \Z^n$, $t_i, t_j \in \Z^m$.

\begin{enumerate} [(i)]

\item  
 We have $\a_u (w v)  = wv=vw = \a_u (vw) $.  Also,
 \begin{align*}
  \a_u (t_i v t_i\i)&= \a_u (t_i) \a_u (v) \a_u  (t_i)\i  \\
 &=\bigg(  (I-M_i)(I-M_1)\i   \u \bigg)   (t_i   v t_i \i)  \bigg( (I-M_i)(I-M_1)\i (-\u) \bigg) \\
 &   =  M_i \v = \a_u ( v ).
 \end{align*}
Next, 
\begin{align*}
&\a_u(t_i t_j t_i\i)  \\
=& \a_u(t_i) \a_u(t_j) \a_u(t_i)\i   \\
=& (I-M_i)(I-M_1)\i \u t_i     \quad   (I-M_j)(I-M_1)\i \u t_j  \quad t_i\i   (-(I-M_i)(I-M_1)\i \u )  \\
=& \bigg((I-M_i)(I-M_1)\i \u      +   M_i(I-M_j)(I-M_1)\i \u\bigg) t_i t_j  t_i\i  (-(I-M_i)(I-M_1)\i \u )  \\
=& \bigg((I-M_i)(I-M_1)\i \u     +    M_i(I-M_j)(I-M_1)\i \u\bigg)  t_j    (-(I-M_i)(I-M_1)\i \u )   \\
=& \bigg((I-M_i)(I-M_1)\i \u     +    M_i(I-M_j)(I-M_1)\i \u   +   M_j\big(-(I-M_i)(I-M_1)\i \u \big)\bigg)t_j   \\
=& (I-M_1)\i\bigg(      (I-M_i)  +  M_i(I-M_j)    +   M_j(-(I-M_i)) \bigg) \u  t_j  \\
=& (I-M_j)(I-M_1)\i \u t_j  = \a_u( t_j ) 
\end{align*}
Hence  $\a_u $ extends to a well-defined homomorphism.  To shows that $\a_u $ is surjective, note that every  $t_i $ in the codomain is the image of $   (I-M_i)(I-M_1)\i (-\u) t_i $.

\item  It is clear that 
$\g_M(wv)= M(\w\v) = M(\v\w) = \g_M(vw) $. Also, since $M \in C(P)$, we have
$$\g_M (t_i v t_i\i) = \g_M (t_i) \g_M (v) \g_M (t_i)\i  = t_i (M \v) t_i\i =   M_i M \v = M M_i \v=\g_M ( M_i \v).$$
Lastly, $\g_M (t_i t_j t_i \i)= \g_M (t_i) \g_M (t_j) \g_M (t_i)\i  = t_i t_j t_i \i= t_j  =  \g_M(t_j).$
Hence  $\g_M $ extends to a well-defined homomorphism. Since  $M \in \glnz$, it follows that $\g_M$ is surjective. \qedhere
\end{enumerate}
\end{proof}
\
We will be using the following notation for the set of automorphism we looked at in the previous lemma.
\begin{enumerate} [(i)]

\item Given $u \in U(P)$, we will denote by  $\a_u $   the automorphism on $P$ defined by  $\a_u(t_1) = u t_1$, $\a_u(v)=v$ for all $v \in \Z^n$.
We will denote by $A(P)$ the set of automorphism in $\A(P)$ that fixes $\Z^n$, i.e.
$A(P) = \{ \a_u \in \Aut(P) 
\mid 
 u \in U(P) \}.$

\item Given $M \in C(P)$, we will denote by $\g_M$   the automorphism  on $P$ defined by  $\g_M(t_i) = t_i$ for all $i$, $\g_M(v)= Mv$ for all $v \in \Z^n$.
We will denote by $\G(P)$ the set of automorphism in $\A(P)$ that fixes all $t_i$'s, i.e $\G(P) = \{ \g_M \in \Aut(P) \mid  M \in C(P) \}.$
\end{enumerate}
In particular, we have the following isomorphisms.

\begin{lem}\label{lem:a=u.g=c}
We have $A(P) \iso U(P)$ and $\G(P) \iso C(P)$.
\end{lem}
 \begin{prop}\label{prop:A+.semidirect}
 Let $P=\Z^n \rtimes \Z^m$ with $\FP= \Z^n$  and suppose $E_1(M_1)=\{0\}$. Then we have $\A (P) = A(P) \rt \G(P)$, i.e.
   \begin{enumerate} [(i)]
 \item $  A(P)\cap \G(P) = \{\id\}$;  
\item $A(P) \nsubg \A(P)$ and $\G(P) \subg \A(P)$;
\item $\A(P)= A(P)\G(P)$.

\end{enumerate}
 \end{prop}

\begin{proof}
For $\p \in \A(P)$, let $M$ and $u$ be the matrix and vector associated $\p$ defined in \cref{what_elt_in_A_plus_looks_like}.
\begin{enumerate} [(ii)]
\item We will check the normality of $A(P)$.  Let $\a_w \in A(P)$ for some $w \in U(P)$. Then for each $v \in \Z^n$, we have
$\p  \a_w  \p\i (v) = \p  \a_w (M\i \v) = \p  (M\i \v) = v$; also
$\p  \a_w  \p\i (t_1) = \p  \a_w (  M\i(-\u) t_1 ) = \p  (  (-M\i \u+\w) t_1  )  =  M(\w) t_1$. 
\item   We have $\p = \a_u \g_M$. \qedhere
\end{enumerate}
\end{proof}

\subsection{ A torsion free abelian-by-abelian polycyclic group with $E_1(M_1)=\{0\}$ has a finite index subgroup of the form $\Z^n \rtimes \Z^m$}
Suppose that $\FP \iso \Z^n $ and $\PFP \iso \Z^m$, we will use the notation introduced in \eqref{label_for_torsion_free_polycyclic}. In addition, the group $P$ has the following presentation:

\begin{equation} \label{p_fp_is_abelian}
P \iso\left\langle 
       \begin{array}{l|cl}
            \FP,            & \FP \iso \Z^n , \quad t_i v t_i\i = M_i \v \text{ for every } v \in \FP; \\
            t_1, \ldots t_m  & t_i t_j t_i \i = u_{i,j} t_j \text{ for } i>j \\                                             
        \end{array}
     \right\rangle
\end{equation}
where each $u_{i,j} \in  \FP$ and   $M_i \in \GL(m,\Z)$ is the matrix corresponding to the restriction of $\p_i$ on $\FP$. For this subsection, we will assume that $1$ is not an eigenvalue of $M_1$.

The following result from \citep[Proposition 2.2]{iso_semi} gives an nice sufficient condition for two semidirect products to be isomorphic.

\begin{lem}\label{inn_isomorph}
    Let $K$ be a group and let $\a \in \Aut(K)$ and $\iota \in \Inn(K)$.  Then
    \[
        K \rtimes_{ \iota \circ \alpha} \Z \cong K \rtimes_{\alpha} \Z.
    \]
\end{lem}

\begin{prop}  \label{finite_index_subgroup_zm_zn}
Suppose that $\FP \iso \Z^n $,  $\PFP \iso \Z^m$ and $E_1(M_1)=\{0\}$. Then there exists $P' \nsubg \f P$ containing $\FP$ such that $\Fit(P')=\Fit(P)$ and $P' \iso \FPP \rt \Z^m$.
\end{prop}

\begin{proof}
We will prove by induction on $i$ that for each $i = 1, \ldots m$ there exists $K_i' \nsubg \f K_i$ containing  $\FP$ and $r_1,\ldots,r_m\in P$ \ such that
\begin{enumerate} [(i)]
\item $\FKiP = \FP$;
\item $K_i' =  \Fit(K_i') \rt    \langle r_1, r_2, \ldots r_i \rangle  \iso \Fit(K_i') \rt \Z^i $; 
 if $Q_1 \in \GL(n,\Z)$ is the matrix corresponding to $r_1$ acting on $\Fit(K_i')$ then  $Q_1 = M_1$.
\end{enumerate}
The case $i=m$ corresponds to the statement of the lemma. Note that since $\PFP$ is abelian, every subgroup of $P$ containing $\FP$ is normal. For $i=1$, it is clear as $K_1 = \FP \rt \Z $. 

Now, suppose the claim holds for $i$ and let $K'_i$ be the subgroup that satisfies the assumptions in the induction hypothesise.   Let $Q_j \in \GL(n,\Z)$ be the matrix corresponding to $r_j$ acting on $\Fit(K_i')$, and let $\V = \{ (I-Q_1)\v \mid v \in  \FP \}$, noting that $\V \subg U(K_i)$ because $Q_1=M_1$. Since $1$ is not an eigenvalue of $Q_1$, we have $\det(I-Q_1) \neq 0$. It follows that $\V \subg\f \Z^n$, and in particular that $V \subg \f U(K_i)$. Therefore, $V$ contains a finite index subgroup $V' \csubg U(K_i)$. Identifying $U(K_i)$ with $A(K_i)$, \ \cref{prop:A+.semidirect}  tells us that
\begin{equation}\label{eq:V'xGamma}
V' \rt \G(K_i) \subg\f A(K_i) \rt \G(K_i)=\A(K_i).
\end{equation} 
Also, by \cref{big_enough_l_K_lZ_is_a_subgroup}, there exists $\ell$ such that $K_i' \rtimes_{\p_{i+1}} \ell\Z_{(i+1)} \subg K_{i+1}$. Since $\PFP$ is abelian, $\p_{i+1} \in \A(K_i)$, and so \eqref{eq:V'xGamma} implies that $\sigma \coloneqq\p_{i+1}^{k\ell} \in V' \rt \G(K_i)$ for some $k\in\N$. Set $K_{i+1}' =  K_i' \rt _{\p_{i+1}} k\ell \Z \cong K_i'\rtimes _{\sigma} \Z$. Note that $\Fit(P)\le K_i'\le K_{i+1}'$ and $K_{i+1}'\normal P$, hence $\Fit(K_{i+1}')=\Fit(P)$ by \cref{normal_subgroup_contains_fit}.

Using the notation in  \cref{existence_of_auto}, since $\sigma\in V' \rt \G(K_i)\le A(K_i) \rt \G(K_i)$, there exist $u\in V' \subeq V $ and $Q \in \bigcap_{j=1} 
^{i}
C_{\GL
(n,\Z)} (Q_j)$  such that $\sigma=\alpha_u\g_Q$.  By definition of $V$, we  have $\u = (I-Q_1) \v$ for some $v \in \FP$.  In particular, $\alpha_u(r_i)=ur_i = (I-Q_i) \v r_i = ( v r_i v \i r_i \i   ) r_i =  v r_i v\i.$ In fact, $\alpha_u(g)=v gv\i$ for all $g \in K_i$, so  $\alpha_u \in \Inn(K_i)$. \cref{inn_isomorph} then implies that
$K_{i+1}' = K_i'\rtimes _{\sigma}  \Z \iso K_i'\rtimes _{ \g_Q} \Z \iso \Fit(K_{i+1}') \rt \Z^{i+1} $
as required.
\end{proof}

\subsection{Polycyclic groups: conclusion of the proof}

\begin{proof}[Proof of \cref{polycyclic_by_finite_group_uda_iff_vn}]
Recall that $G$ has a finite-index normal poly-$\Z$ subgroup $P$. By \cref{normal_subroup_finite_index_contain_normal} and \cref{normal_subgroup_contains_fit}, we can assume that $P/\Fit(P)$ is torsion free. 
The subgroup $P$ has uniformly $(1/d)$-almost flat coset spaces by \cref{uniform_D_alpha_passes_finte_index_subgroup}, and by \cref{thm:nilp.coset} it suffices to prove that $P$ is virtually nilpotent.

We will prove this via induction on $h \coloneqq h(P)$. For $h = 1$, we have $P \iso \Z$, which is nilpotent. 
Now suppose $h \geq 2$, and the statement holds for all poly-$\Z$ groups with Hirsch length less or equal to $h$.  Let $m \coloneqq  h(P/\Fit(P))$; then $P$ is nilpotent if $ m=0$ and not nilpotent if $ m\geq 1$.We will prove by contradiction that \( m = 0 \). Suppose instead that \( m \geq 1 \); we will show that \( P \) does not have uniformly \( \alpha \)-almost flat coset spaces.
 By \cref{write_Polyz_group_with_torsion_free_quotient_as_semi_direct_product}, we have
$ P \iso \FP \rtimes \Z_{(1)} \rtimes \Z_{(2)}\rtimes \ldots \rtimes\Z_{(m)}$.

We first consider the case in which $\Fit(P)$ is not abelian.
In that case, we have $[\FP,\FP]  \csubg$ $\FP \csubg P$, so $[\FP,\FP] \csubg P$. Let $\bar{F} = \FP/ [\FP,\FP]$ and define $$ \bar{P} \coloneqq  \frac{P}{[\FP,\FP]}   \iso \bar{F} \rtimes \Z_{(1)} \rtimes \Z_{(2)}\rtimes \ldots \rtimes\Z_{(m)}. $$
Since $\FP \csubg P$, by correspondence theorem, we have $\bar{F} \csubg \bar{P}$.
Since $\bar{F}$ is nilpotent,  we have $ \Tor(\bar{F}) \csubg \bar{F} \csubg  \bar{P}$. 
 Note that we have $\Tor(\bar{P}) = \Tor(\bar{F})$ , hence   
$$  \frac{\bar{P}}{ \Tor(\bar{P})} \iso \frac{\bar{F}}{\Tor(\bar{F})}  \rtimes \Z_{(1)} \rtimes \Z_{(2)}\rtimes \ldots \rtimes\Z_{(m)}.$$ 
According to \eqref{fitting_of_K_i}, we can see that $\FP \rtimes \Z_{(1)}$ is not virtually nilpotent. By \cref{characterise_nilp_by_z}, we have $\frac{\bar{F}}{\Tor(\bar{F})}  \rtimes \Z_{(1)}$ is not virtually nilpotent. Hence, $\bar{P} / \Tor(\bar{P})$ is not virtually nilpotent. 
 It is clear that $h( \bar{P} / \Tor(\bar{P}) ) < h$. By induction hypothesis, $\bar{P} / \Tor(\bar{P})$ does not have uniformly   $\alpha$-almost flat coset spaces. 
Hence,  $\bar{P}$ does not have uniformly   $\alpha$-almost flat coset spaces, by the same reason, $P$ does not have uniformly   $\alpha$-almost flat coset spaces, contrary to the hypothesis.

We now consider the case in which $\FP$ is abelian. Since $P/\Fit(P)$ has uniformly   $\alpha$-almost flat coset spaces,   it is virtually nilpotent by induction hypothesis.  By \cref{v_nilp_implies_v_abelian}, we can assume that $P/\Fit(P)$ is free abelian.  Hence, $P$ has the following presentation:
$$P \iso \FP \rtimes_{\p_1} \langle t_1 \rangle \rtimes_{\p_2} \langle t_2 \rangle\rtimes \ldots \rtimes _{\p_{m}} \langle t_m\rangle .$$ 
Let   $M_i \in \GL(m,\Z)$ be the matrices corresponding to the restriction of $\p_i$ on $\FP$.  It follows from \cref{big_enough_l_K_lZ_is_a_subgroup} that we can also assume that for each $i$ the only eigenvalue of  $M_i$ that is a root of unity is $1$.  We claim that for every \( i \), \( E_1(M_i) = \{0\} \); that is, \( 1 \) is not an eigenvalue of \( M_i \). We will also prove this claim by contradiction, indeed suppose $E_1(M_i) \neq\{0\}$ for some $i$. According to \cref{commute}, $E_1(M_i) \nsubg P$, so we have 
$$ \frac{P}{E_1(M_i)} 
\iso \frac{\FP}{E_1(M_i)} 
\rt_
{\bar{\p}
_1 }
\langle t_1\rangle 
\rtimes_{\bar{\p}_2} \langle t_2\rangle\rtimes 
\ldots \rtimes _{\bar{\p}_{m}} \langle t_{m}\rangle$$
where each $\bar{\p}_{j}$ is the induced quotient automorphism of $\p_j$. By \cref{torsion_free_quotient}, $P/E_1(M_i)$ is torsion-free, so it is a poly-$\Z$ group with Hirsch length less than $h$. Also, by our assumption, $P/E_1(M_i)$ has uniformly   $\alpha$-almost flat coset spaces, hence it is virtually nilpotent by induction hypothesis. On the one hand,  $\frac{\FP}{E_1(M_i)} 
\rt_
{\bar{\p}
_i }
\langle t_i\rangle$ is a subgroup of $\frac{P}{E_1(M_i)}$,  it implies that $\frac{\FP}{E_1(M_i)} 
\rt_
{\bar{\p}
_i }
\langle t_i\rangle$ is also virtually nilpotent. On the other hand, since $P/\Fit(P)$ is abelian,  $\FP \rtimes _{\p_i} \langle t_i \rangle \nsubg P$. It follows from \eqref{fitting_of_K_i} that $\FP \rtimes _{\p_i} \langle t_i \rangle$ is not virtually nilpotent.  Hence, $M_i$ has an eigenvalue $\l$ with absolute value not equal to $1$. Let $\bar{M_i}$ be the matrix corresponds to ${\bar{\p} _i }$, then clearly $\l$ is also an eigenvalue of $\bar{M_i}$. Hence, $\frac{\FP}{E_1(M_i)} 
\rt_
{\bar{\p}
_i }
\langle t_i\rangle$ is not virtually nilpotent. This gives a contradiction and hence verifies our claim. Finally, by \cref{finite_index_subgroup_zm_zn}, we can assume that $P \iso \FP \rt \Z^m$. By \cref{characteriseing_zm_zn}, we see that $P$ does not have uniformly   $\alpha$-almost flat coset spaces, again contrary to our hypothesis.
\end{proof}

\footnotesize{
\bibliographystyle{abbrv}
\bibliography{Biblio}
}
\end{document}